\documentclass[a4paper,reqno]{amsart}
\usepackage{indentfirst}
\usepackage{amsfonts}
\usepackage{amsthm}
\usepackage{amssymb}
\usepackage{amsmath}
\usepackage{thmtools}
\usepackage{tikz}

\usepackage{todonotes}
\usetikzlibrary{decorations.markings}
\usetikzlibrary{patterns}
\usepackage{enumitem}
\usepackage{setspace}
\usepackage{hyperref}
\usetikzlibrary{matrix,arrows,positioning}

\newtheorem{Lem}{Lemma}[section]
\newtheorem{Prop}[Lem]{Proposition}
\newtheorem*{Def}{Definition}

\theoremstyle{plain}
\newtheorem{Thm}[Lem]{Theorem}

\theoremstyle{definition}
\declaretheorem[numbered=no,name=Example,qed={\lower-0.3ex\hbox{$\triangleleft$}}]{Ex}
\newtheorem*{Rem}{Remark}
\newtheorem*{Rems}{Remarks}

\newcommand{\wt}{\text{\textnormal{wt}}}
\newcommand{\Dim}{\text{\textnormal{\textbf{dim}}}}
\newcommand{\Sym}{\text{\textnormal{Sym}}}

\mathchardef\mhyphen="2D

\makeatletter
\providecommand*{\twoheadrightarrowfill@}{%
  \arrowfill@\relbar\relbar\twoheadrightarrow
}
\providecommand*{\xtwoheadrightarrow}[2][]{%
  \ext@arrow 0579\twoheadrightarrowfill@{#1}{#2}%
}
\makeatother

\def\itemNum$#1${\item $\displaystyle#1$
   \hfill\refstepcounter{equation}(\theequation)}
\def\pser#1{[\![#1]\!]} 

\begin{document}

\title[Self-dual quiver moduli and orientifold DT invariants]{Self-dual quiver moduli and orientifold Donaldson-Thomas invariants}

\author[M.\,B. Young]{Matthew B. Young}
\address{Department of Mathematics\\
The University of Hong Kong\\
Pokfulam, Hong Kong}
\email{myoung@maths.hku.hk}

\date{\today}


\begin{abstract}
Motivated by the counting of BPS states in string theory with orientifolds, we study moduli spaces of self-dual representations of a quiver with contravariant involution. We develop Hall module techniques to compute the number of points over finite fields of moduli stacks of semistable self-dual representations. Wall-crossing formulas relating these counts for different choices of stability parameters recover the wall-crossing of orientifold BPS/Donaldson-Thomas invariants predicted in the physics literature. In finite type examples the wall-crossing formulas can be reformulated in terms of identities for quantum dilogarithms acting in representations of quantum tori. 
\end{abstract}

\maketitle


\setcounter{footnote}{0}

\section*{Introduction}

Representations of quivers and the geometry of their moduli spaces have found applications in  many areas of mathematics, such as the theory of quantum groups, derived categories of coherent sheaves and Donaldson-Thomas theory. Not unrelated, they have also found applications in quantum field theory and string theory.

Quiver moduli were originally constructed by King \cite{king1994} who showed that the definition of stability arising from geometric invariant theory coincides with a purely representation theoretic definition, called slope stability. The latter is modelled on slope stability of vector bundles over curves. More generally, stability of principal bundles over curves, with structure group a classical group $G$, can be understood in terms of slope stability \cite{ramanathan1975}: from the point of view of the vector bundle associated to the defining representation, the potentially destabilizing subbundles are required to be isotropic.

The focus of this paper is the study of moduli spaces of quiver theoretic analogues of $G$-bundles over curves and their relationship with enumerative invariants in string theory with orientifolds. To be more precise, we study moduli spaces of self-dual representations of a quiver with contravariant involution. These representations were introduced by Derksen and Weyman \cite{derksen2002}. While the ordinary representation theory of a quiver assigns a vector space to each node and a linear map to each arrow, the self-dual representation theory in addition endows the vector spaces with orthogonal or symplectic forms and imposes symmetry conditions on the linear maps. From a categorical point of view, the quiver involution can be used to define an exact contravariant endofunctor $S$ of $Rep_k(Q)$ and an isomorphism of functors $\Theta: \mathbf{1}_{Rep_k(Q)} \xrightarrow[]{\sim} S^2$. This makes $Rep_k(Q)$ into an abelian category with duality and the self-dual representations are recovered as its self-dual objects.

We introduce a notion of stability for self-dual quiver representations that is a common generalization of quiver and $G$-bundle stability. This notion coincides with the natural definition of stability arising in geometric invariant theory (Theorem \ref{thm:selfDualStability}). Since the stability parameters in the self-dual theory have less degrees of freedom than their ordinary counterparts, there are in general many strictly semistable self-dual representations. This causes the moduli spaces semistable self-dual representations to be highly singular. Even the stable moduli spaces need not be smooth, having orbifold singularities at non-simple stable self-dual representations. Because of these singularities it will often be more natural to consider moduli stacks of self-dual representations.

A powerful tool in the study of quiver moduli is the Hall algebra. Under assumptions to ensure smoothness, analogous to the coprime assumption for vector bundles over a curve, Hall algebras can be used to compute Poincar\'{e} polynomials of quiver moduli \cite{reineke2003}. This approach uses the Weil conjectures to relate the number of $\mathbb{F}_q$-rational points of quiver moduli to their Poincar\'{e} polynomials. A key r\^{o}le is played by Reineke's integration map, an algebra homomorphism from the Hall algebra to a quantum torus. This map is used to translate categorical identities in the Hall algebra into numerical identities in the quantum torus. Without any smoothness assumptions the same techniques, with moduli stacks in place of moduli spaces, can also be used to study the motivic DT theory of quivers \cite{mozgovoy2013}, \cite{mozgovoy2013c}. Generalizations of the Hall algebra and integration map, some of which remain conjectural, are central to the motivic DT theory of three dimensional Calabi-Yau categories \cite{joyce2012}, \cite{kontsevich2008}, \cite{kontsevich2011}.

The analogue of the Hall algebra for self-dual representations was introduced in \cite{mbyoung2012}. There the self-dual extension structure of the representation category, controlling three term sequences consisting of a self-dual representation, an isotropic subrepresentation and the resulting self-dual quotient, was used to construct a module over the Hall algebra, called the Hall module. In this paper we develop Hall module techniques to study self-dual quiver moduli. The first important result in this direction is the construction of a Hall module integration map in Theorem \ref{thm:sdIntMap}. This is a morphism over the Hall algebra integration map with values in a naturally defined representation of the quantum torus. By modifying arguments of \cite{reineke2003}, in Theorem \ref{thm:recReso} we solve the Harder-Narasimhan recursion for self-dual representations. Applying the Hall module integration map leads immediately to an explicit formula for the number of $\mathbb{F}_q$-rational points of stacks of semistable self-dual representations; see Theorem \ref{thm:countInv}. This provides a quiver theoretic analogue of Laumon and Rapoport's computation of the Poincar\'{e} series of the moduli stack of semistable $G$-bundles over a curve \cite{laumon1996}.

One of the primary motivations of this paper is the development of a mathematical framework for the counting of BPS states in string theory with orientifolds. In the presence of an orientifold, the $D$-brane category $\mathcal{B}$ of the parent theory is endowed with a duality structure $(S, \Theta)$ \cite{diaconescu2007}, \cite{hori2008}. The functor $S$ is the parity functor, sending a $D$-brane to its orientifold image, while $\Theta: \mathbf{1}_{\mathcal{B}} \xrightarrow[]{\sim} S^2$ encodes the signs of the orientifold planes. The $D$-brane configurations in the orientifold theory are precisely the self-dual configurations of the parent theory. Self-dual quiver representations therefore provide a relatively simple example of this set-up. The appearance of orthogonal and symplectic structures reflects the familiar reduction of structure group of Chan-Paton bundles on $D$-branes lying on orientifold planes. Not unrelated, self-dual quiver representations also arise in the study of worldvolume gauge theories on $D$-branes in orientifold backgrounds \cite{douglas1996}.

The cohomology of moduli spaces of semistable $D$-branes is closely related to the BPS states of the theory \cite{denef2011}. Similarly, BPS states in the orientifold theory arise from cohomology of the moduli of orientifold invariant $D$-branes and, for particular theories, should provide an orientifold version of DT invariants. In \cite{walcher2009} it is suggested that real Gromov-Witten invariants are related (via a MNOP type formula) to orientifold DT invariants. Expected properties of orientifold DT invariants for particular models were discussed from a physical perspective in \cite{krefl2010b}. However, a basic definition of the invariants was missing. In this paper, motivated by \cite{kontsevich2008}, \cite{kontsevich2011} we define the orientifold DT series of a quiver with involution as the generating function counting $\mathbb{F}_q$-rational points of stacks of semistable self-dual representations, computed in Theorem \ref{thm:countInv} above. The Hall module formalism leads immediately to a wall-crossing formula, Theorem \ref{thm:sdWallCross}, relating orientifold DT series with different stability parameters. In finite type examples the wall-crossing formulas can be reformulated as quantum dilogarithm identities holding representations of quantum tori. We use these identities to define orientifold DT invariants of finite type quivers; see equation \eqref{eq:oridtFactorization}. These invariants satisfy an orientifold modification of the primitive wall-crossing formula proposed in the physics literature \cite{denef2010}. We take this as strong evidence that the Hall module framework is indeed applicable to the study of orientifold BPS states. In Section \ref{sec:quivPot} we explain how many of the above results can be extended to quivers with potential using equivariant Hall algebras.

In \cite{harvey1998} it was proposed that the space of BPS states in a quantum field theory or string theory with extended supersymmetry has the structure of an algebra, the product of two states encoding their possible bound states. Mathematical models for this algebra include variants of the Hall algebra, most notably its motivic \cite{joyce2007}, \cite{kontsevich2008} and cohomological \cite{kontsevich2011} versions. See also \cite[\S 8]{chuang2014}. Imposing different structures on the physical theory leads to different algebraic structures on its space of BPS states. For example, the space of BPS states in a theory with defects, which can also be thought of as a space of open BPS states, is expected to form a representation of the algebra of BPS states of the theory without defects \cite{gukov2011}. In some examples these open BPS modules are modelled using framed objects of the $D$-brane category \cite{soibelman2014}. The Hall modules used in this paper are different, modelling instead the space of BPS states in a string theory with orientifolds together with an action of the BPS states of the parent theory. These modules are naturally graded by the Grothendieck-Witt group of the $D$-brane category with orientifold duality, an algebraic version of Atiyah's $KR$-theory. This is in agreement with the physical prediction that charges of $D$-branes in orientifold theories are elements of real variants of $K$-theory \cite{witten1998}, \cite{gukov2000}, \cite{hori2008}.

\subsection*{Notation}
Throughout this paper $k$ denotes a fixed ground field. The characteristic of $k$ is assumed to be different from two. We will primarily be interested in the cases $k= \mathbb{C}$ and $k= \mathbb{F}_q$, a finite field with $q$ elements.

If $S$ is a finite set, then $\# S \in \mathbb{Z}_{\geq 0}$ denotes its cardinality.

\subsubsection*{Acknowledgements}
The author thanks Wu-yen Chuang, Zheng Hua, Daniel Krefl, Michael Movshev and Graeme Wilkin for helpful conversations and BIOSUPPORT at the University of Hong Kong for computational support. The author also thanks the Institute for Mathematical Sciences at the National University of Singapore for support and hospitality during the program `The Geometry, Topology and Physics of Moduli Spaces of Higgs Bundles', where this work was completed.

\section{Representation theory of quivers}
\label{sec:quiverReps}

In this section we recall some preliminary material about (self-dual) representations of quivers.

\subsection{Quiver representations}

Let $Q$ be a quiver with finite sets of nodes $Q_0$ and arrows $Q_1$. Denote by $\Lambda_Q=\mathbb{Z} Q_0$ the free abelian group generated by $Q_0$. The monoid of dimension vectors is $\Lambda_Q^+ = \mathbb{Z}_{\geq 0} Q_0$.

A $k$-representation of $Q$ is a finite dimensional $Q_0$-graded vector space $V= \bigoplus_{i \in Q_0} V_i$ together with a linear map $V_i \xrightarrow[]{v_{\alpha}} V_j$ for each arrow $i \xrightarrow[]{\alpha} j \in Q_1$. The dimension vector of $V$ is $\Dim  \, V =\sum_{i \in Q_0} (\dim\, V_i)  i \in \Lambda_Q^+$ and its dimension is $\dim\, V = \sum_{i\in Q_0} \dim\, V_i \in \mathbb{Z}_{\geq 0}$.

The category $Rep_k(Q)$ of $k$-representations of $Q$ is abelian and hereditary. The Euler form of $Rep_k(Q)$ is defined by
\[
\chi(U,V) = \dim \, Hom(U,V) - \dim \, Ext^1(U,V)
\]
and descends to the bilinear form on $\Lambda_Q$ given by
\[
\chi (d, d^{\prime} ) = \sum_{i \in Q_0} d_i d_i^{\prime} - \sum_{i \xrightarrow[]{\alpha} j \in Q_1 } d_i d_j^{\prime}.
\]
The associated skew-symmetric bilinear form on $\Lambda_Q$ is $\langle d, d^{\prime}  \rangle = \chi (d, d^{\prime} ) - \chi (d^{\prime}, d )$.

\subsection{Self-dual quiver representations}

In this section we record some basic material about self-dual representations of a quiver with contravariant involution.

\begin{Def}
An involution $\sigma$ of a quiver $Q$ is a pair of involutions, $Q_0 \xrightarrow[]{\sigma} Q_0$ and $Q_1 \xrightarrow[]{\sigma} Q_1$, such that
\begin{enumerate}
\item if $i \xrightarrow[]{\alpha} j$ is an arrow, then $\sigma(j) \xrightarrow[]{\sigma(\alpha)} \sigma(i)$, and
\item all arrows of the form $ i \xrightarrow[]{\alpha} \sigma(i)$ are fixed by $\sigma$.
\end{enumerate}
\end{Def}

Let $(Q,\sigma)$ be a quiver with involution. There is an induced involution of $\Lambda_Q$, again denoted by $\sigma$, and we write $\Lambda_Q^{\sigma}$ for the subgroup of $\sigma$-invariant dimension vectors. There is a canonical map $H:\Lambda_Q \rightarrow \Lambda_Q^{\sigma}$ given by $d \mapsto d + \sigma(d)$.

 A duality structure on $(Q, \sigma)$ is a pair of functions, $s: Q_0 \rightarrow \{ \pm 1 \}$ and $\tau: Q_1 \rightarrow \{ \pm 1 \}$, such that $s$ is $\sigma$-invariant and $
\tau_{\alpha} \tau_{\sigma(\alpha)} = s_i s_j$ for all arrows $i \xrightarrow[]{\alpha} j$.

\begin{Def}
A self-dual representation of $(Q,\sigma)$ (with respect to a fixed duality structure $(s, \tau)$) is a pair $(M,\langle \cdot, \cdot \rangle)$ consisting of a representation $M$ and a non-degenerate bilinear form $\langle \cdot, \cdot \rangle$ on the total space $\bigoplus_{i \in Q_0} M_i$ such that
\begin{enumerate}
\item the vector spaces $M_i$ and $M_j$ are orthogonal unless and $i =\sigma(j)$,

\item the restriction of the form $\langle \cdot, \cdot \rangle$ to $M_i + M_{\sigma(i)}$ is $s_i$-symmetric,
\[
\langle x,x^{\prime} \rangle = s_i \langle x^{\prime}, x \rangle, \;\;\;\;\;\;\; \forall \, x, x^{\prime} \in M_i + M_{\sigma(i)},
\]
and
\item for all arrows $i \xrightarrow[]{\alpha} j$ the structure maps of $M$ satisfy
\begin{equation}
\label{eq:strSymm}
\langle m_{\alpha} x,x^{\prime} \rangle - \tau_{\alpha}  \langle x, m_{\sigma(\alpha)} x^{\prime} \rangle =0, \;\;\;\;\;\;\; \forall \, x \in M_i, \; x^{\prime} \in M_{\sigma(j)}.
\end{equation}
\end{enumerate}
\end{Def}

Let $M$ be a self-dual representation. If $i=\sigma(i)$ then $M_i$ is endowed with an orthogonal or symplectic form. If instead $i \neq \sigma(i)$, then $M_i \oplus M_{\sigma(i)} \simeq M_i \oplus M_i^{\vee}$ is endowed with the canonical hyperbolic orthogonal or symplectic form. The self-dual representations for $\tau \equiv -1$ and $s \equiv 1$ or $s \equiv -1$ recover the orthogonal or symplectic representations of Derksen and Weyman \cite{derksen2002}. Self-dual representations for more general $(s, \tau)$ were studied in \cite{zubkov2005} where they were called supermixed representations.

There is a categorical interpretation of self-dual representations that will be useful below. Given a duality structure, define an exact contravariant functor $S: Rep_k(Q) \rightarrow Rep_k(Q)$ as follows. At the level of objects, $S(M,m)$ is given by
\[
S(M)_i =M^{\vee}_{\sigma(i)}, \;\;\; S(m)_{\alpha} = \tau_{\alpha} m_{\sigma(\alpha)}^{\vee}.
\]
Here $(-)^{\vee} = Hom_k(-, k)$ is the linear duality functor on the category of finite dimensional vector spaces. Given a morphism $\phi: M \rightarrow M^{\prime}$ with components $\phi_i : M_i \rightarrow M_i^{\prime}$, the morphism $S(\phi):S(M^{\prime}) \rightarrow S(M)$ is defined by its components $S(\phi)_i = \phi^{\vee}_{\sigma(i)}$. Write $\mbox{ev}$ for the canonical evaluation isomorphism from a finite dimensional vector space to its double dual. The assumptions on $(s, \tau)$ imply that $\Theta= \bigoplus_{i \in Q_0} s_i \cdot \mbox{ev}_i$ defines a natural isomorphism from the identity functor $\mathbf{1}_{Rep_k(Q)}$ to the square $S^2$. Moreover, for each representation $U$ the identity $S(\Theta_U)  \circ \Theta_{S(U)} = \mathbf{1}_U$ holds.

The above discussion shows that the triple $(Rep_k(Q), S, \Theta)$ is an example of an abelian category with duality \cite{balmer2005}. In this setting, a self-dual object is a pair $(M, \psi_M)$, or just $M$ for short, consisting of an object $M$ and an isomorphism $\psi_M: M \xrightarrow[]{\sim} S(M)$ satisfying $S(\psi_M) \Theta_M = \psi_M$. An isomorphism $\phi: M \xrightarrow[]{\sim} M^{\prime}$ of self-dual objects is called an isometry if $\psi_M =  S(\phi) \psi_{M^{\prime}} \phi$. We write $M \simeq_S M^{\prime}$ if $M$ and $M^{\prime}$ are isometric. The group of self-isometries of $M$ is denoted $Aut_S(M)$.

Given a self-dual object $M$, the bilinear form $\langle x,x^{\prime} \rangle = \psi_M(x) (x^{\prime})$ gives $M$ the structure of a self-dual representation. This defines an equivalence between the groupoids of self-dual objects and self-dual representations, where the morphisms in each category are the isometries. We will use this equivalence throughout the paper. 

\begin{Ex}
For any representation $U$, the hyperbolic representation on $U$ is the self-dual object $H(U) = (U \oplus S(U), \psi_{H(U)}= \left( \begin{smallmatrix} 0 & \mathbf{1}_{S(U)} \\ \Theta_U & 0 \end{smallmatrix} \right))$.
\end{Ex}

Let $M$ be a self-dual representation with subrepresentation $i: U \hookrightarrow M$. The orthogonal $U^{\perp} \subset M$ is defined to be the kernel of the composition
\[
M \xrightarrow[]{\psi_M} S(M) \xrightarrow[]{S(i)} S(U).
\]
The subrepresentation $U$ is called isotropic if $U \subset U^{\perp}$. In this case the self-dual structure on $M$ induces a canonical self-dual structure on the quotient $U^{\perp} \slash U$, denoted by $M /\!/ U$.

For any representation $U$ and $i \geq 0$ the pair $(S, \Theta)$ gives $Ext^i(S(U),U)$ the structure of a representation of $\mathbb{Z}_2$. Decompose this representation into its trivial and sign subrepresentations,
\[
Ext^i(S(U),U)= Ext^i(S(U),U)^{S} \oplus Ext^i(S(U),U)^{-S},
\]
and define
\[
\mathcal{E}(U) = \dim \,  Hom(S(U),U)^{-S} - \dim \, Ext^1(S(U),U)^S.
\]
The function $\mathcal{E}$ will play the r\^{o}le of the Euler form for the category with duality $(Rep_k(Q), S, \Theta)$.

It was shown in \cite[Proposition 3.3]{mbyoung2012} that $\mathcal{E}(U)$ depends only on $\Dim \, U$ and so defines a function $\mathcal{E}: \Lambda_Q \rightarrow \mathbb{Z}$. Explicitly, from \textit{loc. cit.} we have
\begin{equation}
\label{eq:modEuler}
\mathcal{E}(d)=\sum_{i \in Q_0^{\sigma}} \frac{d_i(d_i -s_i)}{2}  + \sum_{i \in Q_0^+} d_{\sigma(i)} d_i - \sum_{\sigma(i) \xrightarrow[]{\alpha} i \in Q_1^{\sigma}} \frac{d_i(d_i + \tau_{\alpha} s_i)}{2} -\sum_{i \xrightarrow[]{\alpha} j  \in Q_1^+} d_{\sigma(i)} d_j.
\end{equation}
Here $Q_0=Q_0^+ \sqcup Q_0^{\sigma}  \sqcup Q_0^-$ is a partition with $Q_0^{\sigma}$ consisting of the nodes fixed by $\sigma$ and $\sigma(Q_0^+)= Q_0^-$. The partition of $Q_1$ is analogous.

Below we will also use the function $\tilde{\mathcal{E}}: \Lambda_Q \rightarrow \mathbb{Z}$ defined by $\tilde{\mathcal{E}}(d) = \mathcal{E}(d) - \mathcal{E}(\sigma(d))$.

\section{Moduli spaces of self-dual quiver representations}
\label{sec:stability}

In this section we introduce a notion of stability for self-dual representations and use geometric invariant theory to construct moduli spaces of self-dual representations.

\subsection{\texorpdfstring{$\sigma$}{sigma}-Stability}

Fix an element $\theta \in \Lambda_Q^{\vee} = Hom_{\mathbb{Z}}(\Lambda_Q, \mathbb{Z})$, called a stability. The slope of a non-zero representation $U$ with respect to $\theta$ is
\[
\mu(U)= \frac{\theta( U)}{\dim \, U} \in \mathbb{Q}.
\]
Here $\theta(U)$ is shorthand for $\theta( \Dim\, U)$.

\begin{Def}[\cite{king1994}]
A representation $U$ is semistable if $\mu(V) \leq \mu(U)$ for all non-zero subrepresentations $V \subsetneq U$. If this inequality is strict for all such $V$, then $U$ is called stable.
\end{Def}

Let $(Q, \sigma)$ be a quiver with involution. Denote by $\sigma^*$ the induced involution of $\Lambda_Q^{\vee}$.

\begin{Def}
A stability $\theta \in \Lambda_Q^{\vee}$ is called $\sigma$-compatible if $\sigma^* \theta = -\theta$.
\end{Def}

If $\theta$ is $\sigma$-compatible, then $\mu(S(U))= - \mu(U)$ for all representations $U$. In particular, the slope of a self-dual representation is necessarily zero.

\begin{Lem}
\label{lem:symmStab}
Let $\theta$ be a $\sigma$-compatible stability. A representation $U$ is semistable (stable) if and only if $S(U)$ is semistable (respectively, stable).
\end{Lem}
\begin{proof}
The representation $U$ is semistable if and only if $\mu(U) \leq \mu(W)$ for all quotients $U \twoheadrightarrow W$. Since the functor $S$ defines a bijection between quotients of $U$ and subobjects of $S(U)$, the $\sigma$-compatibility of $\theta$ implies the statement for semistability. The argument for stability is identical.
\end{proof}

The following definition is motivated by the stability of principal bundles over a curve with classical structure groups \cite{ramanathan1975}.

\begin{Def}
A self-dual representation $M$ is $\sigma$-semistable if $\mu(V) \leq \mu(M)$ for all non-zero isotropic subrepresentations $V \subset M$. If this inequality is strict for all such $V$, then $M$ is called $\sigma$-stable.\end{Def}

\textit{A priori}, $\sigma$-semistability is strictly stronger than semistability. However, we have the following result. See \cite[Proposition 4.2]{ramanan1980} for the analogous statement for $G$-bundles over curves.

\begin{Prop}
\label{prop:sdSemiStab}
A self-dual representation is $\sigma$-semistable if and only if it is semistable as an ordinary representation.
\end{Prop}
\begin{proof}
Suppose that $M$ is $\sigma$-semistable but not semistable. Let $i:U \hookrightarrow M$ be the strongly contradicting semistability subrepresentation, that is, the subrepresentation with maximal slope and maximal dimension among such subrepresentations. Then $U$, and by Lemma \ref{lem:symmStab} also $S(U)$, is semistable with
\[
\mu(S(U)) < \mu(M) < \mu(U).
\]
This implies that the composition
\[
U \xrightarrow[]{i} M \xrightarrow[]{\psi_M} S(M) \xrightarrow[]{S(i)} S(U)
\]
vanishes, being a map between semistable representations of strictly decreasing slope. But then $U$ is isotropic, contradicting the supposed $\sigma$-semistability of $M$. The converse is immediate.
\end{proof}

From now on we will refer to $\sigma$-semistability simply as semistability.

\begin{Prop}
\label{prop:sdHNFilt}
Every self-dual representation $M$ has a unique filtration
\[
0 = U_0 \subset U_1 \subset \cdots \subset U_r \subset M
\]
by isotropic subrepresentations such that the subquotients $U_1 \slash U_0, \dots, U_r \slash U_{r-1}$ are semistable, the self-dual quotient $M /\!/ U_r$ is zero or semistable and the slopes satisfy
\[
\mu( U_1 \slash U_0) > \mu (U_2 \slash U_1) > \cdots > \mu( U_r \slash U_{r-1}) > 0.
\]
\end{Prop}

\begin{proof}
If $M$ is semistable, then $0 \subset M$ is the desired filtration. So assume that $M$ is not semistable and proceed by induction on $\dim \, M$. The case $\dim \, M=1$ is vacuous since $M$ is semistable. Let $U_1 \subset M$ be the (non-zero) strongly contradicting semistability subrepresentation, which is isotropic by the proof of Proposition \ref{prop:sdSemiStab}. The inductive hypothesis implies that $M /\!/ U_1$ has a unique filtration with the required properties. Pulling this back by the quotient morphism $U_1^{\perp} \twoheadrightarrow M /\!/ U_1$ gives the desired filtration of $M$. Uniqueness follows from the uniqueness of the strongly contradicting semistability subrepresentation.
\end{proof}

The filtration given in Proposition \ref{prop:sdHNFilt} is called the $\sigma$-Harder-Narasimhan (HN) of $M$. In fact, the $\sigma$-HN filtration coincides with the positive half (according to slope) of the HN filtration of $M$, viewed as an ordinary representation.

We now turn to $\sigma$-stability. Recall that a stable $k$-representation $U$ is called absolutely stable if its base change $U\otimes_k \overline{k}$ is a stable $\overline{k}$-representation. Absolutely $\sigma$-stable representations are defined analogously. A $\sigma$-stable representation is called regular if it is also stable as an ordinary representation.

\begin{Lem}
\label{lem:ramananLem}
Any $\sigma$-stable self-dual representation is isometric to an orthogonal direct sum of regular $\sigma$-stable self-dual representations.
\end{Lem}

\begin{proof}
The proof of \cite[Proposition 4.2, Remark 4.3(ii)]{ramanan1980} can be applied without change in the quiver setting with no restriction on the ground field.
\end{proof}

Let $M$ be a polystable $k$-representation, that is, a direct sum of stable representations of the same slope. Then $M \otimes_k \overline{k}$ is polystable and $Aut( M \otimes_k \overline{k})$ is a product of general linear groups. If $M$ admits a self-dual structure, then $Aut_S( M \otimes_k \overline{k})$ is a product general linear, symplectic and orthogonal groups; suppose there are $r$ orthogonal factors. If the field $k$ is finite, then the Galois cohomology of the isometry group is
\[
H^1(k, Aut_S( M \otimes_k \overline{k})) \simeq \mathbb{Z}_2^r,
\]
each factor being identified with the choice of discriminant $\varepsilon_i \in k^{\times} \slash k^{\times 2} \simeq \mathbb{Z}_2$ of the corresponding orthogonal form. It follows that there are $2^r$ $k$-forms of the self-dual representation $M \otimes_k \overline{k}$. The $k$-form associated to $\varepsilon \in H^1(k, Aut_S( M \otimes_k \overline{k}))$ will be denoted $M^{\varepsilon}$. See \cite[Chapter III]{serre2002} for background results on Galois cohomology.

\begin{Prop}
\label{prop:sdStab}
Assume that $k$ is finite or algebraically closed. A self-dual representation is $\sigma$-stable if and only if it is isometric to an orthogonal direct sum of the form $\bigoplus_l M_l^{\oplus m_l, \varepsilon_l}$, where $M_l$ are pairwise non-isomorphic regular $\sigma$-stable representations and $m_l=1$ or (if $k$ is finite) $m_l=2$ and $M_l^{\oplus 2, \varepsilon_l}$ is non-hyperbolic.
\end{Prop}

\begin{proof}
Suppose first that $k$ is finite. Let $M$ be $\sigma$-stable. Note that $M$ has no hyperbolic summands. By Lemma \ref{lem:ramananLem} there are pairwise non-isomorphic regular $\sigma$-stable representations $M_l$ such that $M= \bigoplus_l M_l^{\oplus m_l, \varepsilon_l}$. If $m_l \geq 3$ for some $l$, then there exists a discriminant $\varepsilon^{\prime}_l$ so that
\[
M_l^{\oplus m_l, \varepsilon_l} \simeq_S H(M_l) \oplus M_l^{\oplus (m_l-2), \varepsilon^{\prime}_l},
\]
contradicting $\sigma$-stability of $M$. It follows that $m_l=1$ or $m_l= 2$ and $\varepsilon_l$ is non-hyperbolic.

Conversely, consider $M= \bigoplus_l M_l^{\oplus m_l, \varepsilon_l}$ as in the statement of the proposition. A slope zero subrepresentation $U \subset M$ is necessarily a direct sum of copies of the $M_l$. The assumptions on $m_l$ and $\varepsilon_l$ imply that $M_l$ does not appear as an isotropic subrepresentation of $M_l^{\oplus m_l, \varepsilon_l}$. Hence $U$ is not isotropic and $M$ is $\sigma$-stable.

When $k=\overline{k}$ the same proof applies. In this case the isometry $M_l^{\oplus 2} \simeq_S H(M_l)$ implies that $m_l = 1$ for all $l$. 
\end{proof}

\begin{Prop}
\label{prop:sdStabAut}
Assume that $k$ is finite or algebraically closed.
\begin{enumerate}
\item A $\sigma$-stable representation $M$ is absolutely $\sigma$-stable if and only if, in the notation of Proposition \ref{prop:sdStab}, each $M_l$ is absolutely stable and $m_l=1$.

\item If $M$ is an absolutely $\sigma$-stable representation with $r$ regular summands, then $Aut_S(M) \simeq \mathbb{Z}_2^r$.
\end{enumerate}
\end{Prop}

\begin{proof}
When $k = \overline{k}$ the first statement is Proposition \ref{prop:sdStab}. So assume that $k$ is finite and, in the notation of Proposition \ref{prop:sdStab}, write a $\sigma$-stable representation as $M=\bigoplus_l M_l^{\oplus m_l, \varepsilon_l}$. Then
\[
M \otimes_k \overline{k} \simeq_S \bigoplus_l (M_l\otimes_k \overline{k})^{\oplus m_l} .
\]
By Proposition \ref{prop:sdStab} $M \otimes_k \overline{k}$ is $\sigma$-stable if and only if $m_l =1$ and the summands are pairwise non-isomorphic regular $\sigma$-stable representations. By Hilbert's Theorem 90, the summands $M_l\otimes_k \overline{k}$ are pairwise non-isomorphic if and only if the $M_l$ are. Also, $M_l\otimes_k \overline{k}$ is regular $\sigma$-stable if and only if $M_l$ is absolutely stable. This proves the first part of the proposition.

For the second statement, writing $M = \bigoplus_{l=1}^r M_l$ as above (omitting $\varepsilon_l$ from the notation), Schur's lemma gives $End(M) = \bigoplus_{l=1}^r End(M_l)$. 
Hence $Aut_S(M) =\bigoplus_{l=1}^r Aut_S(M_l)$. Since each $M_l$ is absolutely stable, $End(M_l) \simeq k$ and $Aut_S(M_l) \simeq \mathbb{Z}_2$. The statement follows.
\end{proof}

\subsection{GIT stability and moduli spaces}

The affine variety of $k$-representations of $Q$ of dimension vector $d \in \Lambda_Q^+$ is $R_d = \bigoplus_{i \xrightarrow[]{\alpha} j} Hom_k(k^{d_i}, k^{d_j})$. The algebraic $k$-group $GL_d = \prod_{i \in Q_0} GL_{d_i}$ acts by simultaneous base change on $R_d$ and its orbits are in bijection with the set of isomorphism classes of representations of dimension vector $d$.

Assume that $k$ is algebraically closed. Fix $d \in \Lambda_Q^{\sigma, +}$ and assume that $d_i$ is even if $i \in Q_0^{\sigma}$ with $s_i=-1$. Up to isometry, there is a unique self-dual structure $\langle \cdot, \cdot \rangle$ on the trivial representation of dimension vector $d$. Denote by $R_d^{\sigma} \subset R_d$ the subspace of structure maps satisfying equation \eqref{eq:strSymm} with respect to $\langle \cdot, \cdot \rangle$. Explicitly,
\[
R_d^{\sigma} \simeq \bigoplus_{i \xrightarrow[]{\alpha} j \in Q_1^+} Hom_k(k^{d_i}, k^{d_j} ) \oplus \bigoplus_{i \xrightarrow[]{\alpha} \sigma(i) \in Q_1^{\sigma}} Bil^{s_i \tau_{\alpha}}(k^{d_i}).
\]
Here $Bil^{\epsilon}(V)$ denotes the vector space of symmetric ($\epsilon=1$) or skew-symmetric ($\epsilon=-1$) bilinear forms on a vector space $V$. The isometry group of $\langle \cdot, \cdot \rangle$ is the reductive $k$-group $G_d^{\sigma} = \prod_{i \in Q_0^+} GL_{d_i}  \times \prod_{i \in Q_0^{\sigma}} G_{d_i}^{s_i}$ where $G_{d_i}^{s_i}$ is an orthogonal or symplectic group:
\[
G_{d_i}^{s_i} = \left\{ \begin{array}{ll} O_{d_i}, & \mbox{ if } s_i = 1 \\  Sp_{d_i}, & \mbox{ if } s_i = -1.
 \end{array} \right.
\]
The group $G_d^{\sigma}$ acts on $R_d^{\sigma}$ through the embedding $G_d^{\sigma} \hookrightarrow  GL_d$ given on factors by $G_{d_i}^{s_i} \hookrightarrow GL_{d_i}$ for $i \in Q_0^{\sigma}$ and
\[
GL_{d_i} \rightarrow GL_{d_i} \times GL_{d_{\sigma(i)}}, \;\;\;\;\;\;\; g_i \mapsto (g_i, (g_i^{-1})^T)
\]
for $i \in Q_0^+$. Isometry classes of self-dual representations of dimension vector $d$ are in bijection with the $G_d^{\sigma}$-orbits of $R_d^{\sigma}$.

If $k$ is finite, the bilinear form $\langle \cdot, \cdot \rangle$ need not be uniquely defined. Indeed, for each $i \in Q_0^{\sigma}$ with $s_i=1$ there are two inequivalent choices for the restriction of $\langle \cdot, \cdot \rangle$ to $M_i$, labelled by a discriminant $\varepsilon_i \in \mathbb{Z}_2$. Fixing a choice $\varepsilon$ of these discriminants, there is an associated algebraic group $G_d^{\sigma, \varepsilon}$ and a $G_d^{\sigma, \varepsilon}$-representation $R_d^{\sigma, \varepsilon}$. As a vector space $R_d^{\sigma, \varepsilon}$ is independent of $\varepsilon$. Isometry classes of self-dual representations of dimension vector $d$ are in bijection with the $G_d^{\sigma,\varepsilon}$-orbits of $R_d^{\sigma,\varepsilon}$ as $\varepsilon$ varies over all choices.

Each stability $\theta \in \Lambda_Q^{\vee}$ defines a character
\[
\chi_{\theta}: GL_d \rightarrow k^{\times}, \;\;\;\;\;\;\;
(\{g_i\}_{i\in Q_0}) \mapsto  \prod_{i \in Q_0} (\det g_i)^{-\theta_i}
\]
and by restriction also a character of $G_d^{\sigma}$. Stabilities satisfying $\sigma^* \theta = \theta$ restrict to the trivial character of the identity component of $G_d^{\sigma}$. In fact up to a factor of one half, which is irrelevant for GIT, the characters of the identity component can be identified with the $\sigma$-compatible stabilities.

Recall the definition of stability in GIT \cite{mumford1994}, \cite{king1994}. Assume that $k = \overline{k}$ and let $V$ be a representation of a (not necessarily connected) reductive group $G$ with kernel $\Delta \subset G$. Fix a character $\chi: G \rightarrow k^{\times}$.

\begin{Def}
A point $v \in V$ is $\chi$-semistable if there exists $n \geq 1$ and
\[
f \in k[V]^{G,\chi^n} = \{ h \in k[V]  \mid  h(g \cdot v^{\prime}) = \chi(g)^n h( v^{\prime} ), \; \forall g \in G, \; v^{\prime} \in V \}
\]
such that $f(v) \neq 0$. If, in addition, the stabilizer $\mbox{Stab}_{G\slash \Delta}( v)$ is finite and the action of $G$ on $\{ v^{\prime} \in V \mid f(v^{\prime}) \neq 0 \}$ is closed, then $v$ is called $\chi$-stable.
\end{Def}

The $\chi$-(semi)stable points for the action of $G$ and its identity component coincide \cite[Proposition 1.15]{mumford1994}. In particular, we can apply the usual Hilbert-Mumford criterion to test stability, regardless of the connectivity of $G$.

\begin{Thm}
\label{thm:selfDualStability}
Assume that $k = \overline{k}$ and let $\theta$ be a $\sigma$-compatible stability. A self-dual representation $M \in R_d^{\sigma}$ is $\sigma$-(semi)stable if and only if it is $\chi_{\theta}$-(semi)stable.
\end{Thm}
\begin{proof}
We follow the strategy of \cite[\S 3]{king1994} where the analogous statement for ordinary representations is proven. We will prove the statement for stability. The argument for semistability is the same.

Given $M \in R_d^{\sigma}$ and a cocharacter $\lambda: k^{\times} \rightarrow G^{\sigma}_d$ define
\[
M_i^a = \left\{ x \in M_i  \mid  \lambda(z) \cdot x = z^a x, \;\;\; \forall z \in k^{\times} \right\}, \;\;\; a \in \mathbb{Z}, \; i \in Q_0.
\]
For each arrow $i \xrightarrow[]{\alpha} j$ the structure map $m_{\alpha}$ decomposes into a collection of linear maps $m_{\alpha}^{a,b} : M_i^a \rightarrow M_j^b$ satisfying $\lambda(z) \cdot m_{\alpha}^{a,b} = z^{b-a} m_{\alpha}^{a,b}$. Then $\displaystyle \lim_{z \rightarrow 0} \lambda(z) \cdot M$ exists if and only if $m_{\alpha}^{a,b}=0$ for all $a > b$ and $\alpha \in Q_1$. The latter condition is equivalent to the direct sum
\[
M_{(w)} = \bigoplus_{i \in Q_0} \bigoplus_{ a \geq w} M_i^a
\]
being a subrepresentation of $M$ for each $w \in \mathbb{Z}$. Then $\{M_{(w)} \}_{w \in \mathbb{Z}}$ is a decreasing filtration of $M$ stabilizing at $0$ for $w \gg 0$ and at $M$ for $ w \ll 0$.

Let $x \in M_i^a$ and $x^{\prime} \in M_{\sigma(i)}^b$. Since $\lambda$ acts by isometries we have
\[
\langle x, x^{\prime} \rangle = \langle \lambda(z) x, \lambda(z) x^{\prime} \rangle = z^{a+b} \langle x, x^{\prime} \rangle.
\]
Therefore $\langle x, x^{\prime} \rangle = 0$ whenever $a \neq -b$, implying $M_{(w)}^{\perp} = M_{(-w+1)}$. In particular, $M_{(w)}$ is isotropic if $w >0$.

Writing $( \cdot, \cdot )$ for the canonical pairing between characters and cocharacters, we compute
\begin{eqnarray*}
( \chi_{\theta} , \lambda ) &=&  \sum_{w \in \mathbb{Z}} \theta(M_{(w)}) \\
&=& \sum_{w >0}  \theta(M_{(-w+1)}) + \sum_{w > 0 } \theta(M_{(w)}) \\ 
&=& \sum_{w >0} \left( \theta (M /\!/ M_{(w)}) + \theta(M_{(w)}) \right)+ \sum_{w > 0 } \theta(M_{(w)}) \\ 
&=& 2 \sum_{w > 0 } \theta(M_{(w)}).
\end{eqnarray*}
In the last line we used that $\theta$ vanishes on $\Lambda_Q^{\sigma}$.

If $M$ is $\sigma$-stable, from the previous calculation we see that $( \chi_{\theta} , \lambda )  < 0$ for all cocharacters $\lambda$. By the Hilbert-Mumford criterion (in the form of \cite[Proposition 2.5]{king1994}) $M$ is $\chi_{\theta}$-stable. Conversely, suppose that $M$ is $\chi_{\theta}$-stable. A non-zero isotropic subrepresentation $U \subset M$ defines a filtration
\begin{equation}
\label{eq:isoFilt}
U \subset U^{\perp} \subset M.
\end{equation}
There exists a cocharacter $\lambda: k^{\times} \rightarrow G^{\sigma}_d$ whose limit $\lim_{z \rightarrow 0} \lambda(z) \cdot M$ exists and whose associated filtration is \eqref{eq:isoFilt}; take $\lambda$ to have weight $-1$ on $U$, weight zero on a vector space complement of $U$ in $U^{\perp}$, and weight $1$ on a complement of $U^{\perp}$. The Hilbert-Mumford criterion implies $2 \theta(U)= ( \chi_{\theta},\lambda ) < 0$, proving that $M$ is $\sigma$-stable.
\end{proof}

For each $\sigma$-compatible stability $\theta$ and dimension vector $d \in \Lambda_Q^{\sigma, +}$, define the moduli space of semistable self-dual representations as the GIT quotient
\[
\mathfrak{M}^{\sigma, \theta}_d =  \mbox{Proj} \left( \bigoplus_{n \geq 0} k[R_d^{\sigma}]^{G_d^{\sigma}, \chi_{\theta}^n }\right).
\]
It is an irreducible normal quasi-projective variety parameterizing $S$-equivalence classes of semistable self-dual representations. More precisely, each semistable self-dual representation $M$ has a $\sigma$-Jordan-H\"{o}lder filtration
\[
0=U_0 \subset U_1 \cdots \subset U_r \subset M,
\]
with subquotients $U_1 \slash U_0, \dots, U_r \slash U_{r-1}$ stable of slope zero and self-dual quotient $M /\!/ U_r$ zero or $\sigma$-stable. The associated graded self-dual representation is
\[
Gr_S(M) =\bigoplus_{i=1}^r H(U_i \slash U_{i-1}) \oplus M /\!/ U_r.
\]
Two semistable self-dual representations $M_1$ and $M_2$ are $S$-equivalent if $Gr_S(M_1) \simeq_S Gr_S(M_2)$. Using this and the self-dual generalization of \cite[Theorem 2.6]{derksen2002} we conclude that the forgetful map $\mathfrak{M}_d^{\sigma,\theta} \rightarrow \mathfrak{M}_d^{\theta }$ to the moduli space of ordinary representations is injective.

There is an open subvariety $\mathfrak{M}_d^{\sigma, \theta \mhyphen st} \subset \mathfrak{M}_d^{\sigma,\theta }$ parameterizing isometry classes of $\sigma$-stable representations. While $\mathfrak{M}_d^{\sigma, \theta \mhyphen st}$ in general has orbifold singularities (see Proposition \ref{prop:sdStabAut}) the regular $\sigma$-stable representations are smooth points. If non-empty, $\mathfrak{M}_d^{\sigma,\theta \mhyphen st}$ is of dimension $-\mathcal{E}(d)$. This can be seen either by direct calculation or by identifying the tangent space of the moduli stack of self-dual representations at $M$ with $Ext^1(M,M)^S$ and the infinitesimal isometries of $M$ with $Hom(M,M)^{-S}$.

\begin{Ex}
Let $Q$ be the following orientation of the $A_{2n}$ Dynkin diagram:
\vspace{5pt}
\begin{center}
\begin{tikzpicture}[thick,scale=.4,decoration={markings,mark=at position 0.6 with {\arrow{>}}}]
\draw[loosely dotted] (-2,0) to  (0,0);
\draw[postaction={decorate}] (0,0) to  (2,0);
\draw[postaction={decorate}] (2,0) to  (4,0);
\draw[postaction={decorate}] (4,0) to  (6,0);
\draw[loosely dotted] (6,0) to  (8,0);

\fill (-2,0) circle (4pt);
\fill (0,0) circle (4pt);
\fill (2,0) circle (4pt);
\fill (4,0) circle (4pt);
\fill (6,0) circle (4pt);
\fill (8,0) circle (4pt);

\draw (-2,-0.8) node {\footnotesize $\mhyphen n$};
\draw (-.1,-0.8) node {\footnotesize $ \mhyphen 2$};
\draw (1.9,-0.8) node {\footnotesize $ \mhyphen 1$};
\draw (4,-0.8) node {\footnotesize $1$};
\draw (6,-0.8) node {\footnotesize $2$};
\draw (8,-0.8) node {\footnotesize $n$};
\end{tikzpicture}
\end{center}
The involution swaps nodes $i$ and $-i$, fixes the middle arrow and swaps the remaining arrows. If a representation is orthogonal (symplectic) then, along with other conditions, the map assigned to the middle arrow is skew-symmetric (respectively, symmetric).

For stability $\theta_i = -i$ the stable representations coincide with the indecomposable representations, which are in bijection with the positive roots of $A_{2n}$. There are no $\sigma$-stable orthogonal representations. The semistable orthogonal representations are hyperbolic sums of $\sigma$-symmetric indecomposables. The symplectic case depends on the ground field. When $k = \overline{k}$ the regular $\sigma$-stable symplectic representations are precisely the $\sigma$-symmetric indecomposables. When $k$ is finite each $\sigma$-symmetric indecomposable admits two distinct symplectic structures and this give all the regular $\sigma$-stables. There is also a unique $\sigma$-stable symplectic structure on the twofold direct sum of each $\sigma$-symmetric indecomposable. After base change to $\overline{k}$ this representation is hyperbolic and so is not absolutely $\sigma$-stable.
\end{Ex}

\begin{Ex}
Let $K_n$ be the $n$-Kronecker quiver
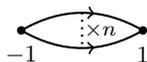
\[
\begin{tikzpicture}[thick,scale=.8,decoration={markings,mark=at position 0.6 with {\arrow{>}}}]
\draw[postaction={decorate}] (0,0) [out=30,in=150]  to  (2,0);
\draw[postaction={decorate}] (0,0) [out=-30,in=210]  to  (2,0);

\draw[dotted] (1,.3) to (1,-.3);

\draw (0,-.4) node {\small $-1$};
\draw (2,-.4) node {\small $1$};

\draw (1.3,0) node {\footnotesize $ \times n$};

\fill (0,0) circle (2pt);
\fill (2,0) circle (2pt);
\end{tikzpicture}
\]
with the involution that swaps the nodes and fixes the arrows. Symplectic representations have symmetric structure maps. Fix the stability $\theta_i=-i$ and identify $G_d^{\sigma}$ with $GL_{d_1}$. A symplectic representation of dimension vector $(1,1)$ is semistable if and only if it is $\sigma$-stable if and only if not all of its structure maps are zero. Hence $\mathfrak{M}^{\mathfrak{sp},\theta }_{(1,1)} \simeq \mathbb{P}^{n-1}$, arising as the coarse moduli space of a $\mathbb{Z}_2$-gerbe over $\mathbb{P}^{n-1}$. For $n>2$, $\mathfrak{M}_{(2,2)}^{\mathfrak{sp}, \theta }$ is in general singular.

Moduli spaces of $K_2$-representations can be described explicitly. For each $d \geq 1$, using Jordan-H\"{o}lder filtrations and taking symmetric products gives isomorphisms
\begin{equation}
\label{eq:symmPower}
\mathfrak{M}^{\mathfrak{sp},\theta }_{(d,d)} \simeq \Sym^d \, \mathfrak{M}^{\mathfrak{sp},\theta }_{(1,1)} \simeq \mathbb{P}^d.
\end{equation}
From Proposition \ref{prop:sdStab} $\mathfrak{M}^{\mathfrak{sp},\theta \mhyphen st}_{(d,d)}$ is the complement of the big diagonal in \eqref{eq:symmPower}. This contrasts the situation for ordinary representations, where $\mathfrak{M}_{(d,d)}^{\theta \mhyphen st}$ is empty if $d >1$.
\end{Ex}

\section{Orientifold Donaldson-Thomas theory of a quiver}

In this section we introduce the orientifold Donaldson-Thomas series of a quiver with duality structure. We use Hall algebras and their representations to study the basic properties of these series, including their wall-crossing.

\subsection{Quantum torus and coistropics}

Fix a finite field $k = \mathbb{F}_q$ of odd characteristic and let $Q$ be a quiver. The quantum torus $\hat{\mathbb{T}}_Q$ attached to $Rep_{\mathbb{F}_q}(Q)$ is the $\mathbb{Q}(q^{1/2})$-vector space with topological basis $\{ x_d \}_{d \in \Lambda_Q^+}$ and multiplication
\begin{equation}
\label{eq:quantumTorus}
x_d \cdot x_{d^{\prime}} = q^{\frac{1}{2} \langle d, d^{\prime} \rangle} x_{d+d^{\prime}}.
\end{equation}
The algebra $\hat{\mathbb{T}}_Q$, or rather the algebra generated by $\{x_d\}_{d \in \Lambda_Q}$ with the same multiplication, is a quantization of the Poisson algebra of regular functions on the algebraic torus $T_Q=\Lambda_Q^{\vee} \otimes_{\mathbb{Z}} \mathbb{C}^{\times}$ with Poisson structure determined by the skew-symmetrized Euler form $\langle \cdot, \cdot \rangle$ \cite{kontsevich2008}.

In the self-dual setting, let $\hat{\mathbb{S}}_Q$ be the $\mathbb{Q}(q^{1/2})$-vector space with topological basis $\{ \xi_e \}_{e \in \Lambda_Q^{\sigma,+}}$. Define an action of $\hat{\mathbb{T}}_Q$ on $\hat{\mathbb{S}}_Q$ by the formula
\begin{equation}
\label{eq:quantumModule}
x_d \star \xi_e = q^{\frac{1}{2} ( \langle d, e \rangle + \tilde{\mathcal{E}}(e) ) } \xi_{H(d) + e}.
\end{equation}
Using the identity
\[
\mathcal{E}(d+d^{\prime}) = \mathcal{E}(d) + \mathcal{E}(d^{\prime}) + \chi( \sigma(d), d^{\prime} ).
\]
it is verified that this gives $\hat{\mathbb{S}}_Q$ the structure of a $\hat{\mathbb{T}}_Q$-module.

The module $\hat{\mathbb{S}}_Q$ has the following geometric interpretation.\footnote{Again, for this interpretation we should use the module generated by $\{\xi_e\}_{e \in \Lambda_Q^{\sigma}}$.} The involution $\sigma : \Lambda_Q \rightarrow \Lambda_Q$ induces an anti-Poisson involution $\sigma^* :T_Q \rightarrow T_Q$ whose fixed point locus is a coisotropic subtorus $T_Q^{\sigma} \subset T_Q$. The algebra of regular functions on $T_Q^{\sigma}$, and more generally the space of sections of a vector bundle over $T_Q^{\sigma}$, inherits the structure of a $\mathbb{C}[T_Q]$-module. From this point of view, $\hat{\mathbb{S}}_Q$ is a quantization of the $\mathbb{C}[T_Q]$-module of sections of the trivial vector bundle of rank $2^{\# Q_0^{\sigma}}$ over $T_Q^{\sigma}$.

\subsection{Hall algebras, Hall modules and integration maps}

Let $\mathcal{H}_Q$ be the Hall algebra of $Rep_{\mathbb{F}_q}(Q)$ \cite{ringel1990}. Its underlying $\mathbb{Q}$-vector space is generated by symbols $[U]$ indexed by isomorphism classes of $\mathbb{F}_q$-representations of $Q$ and its multiplication is
\[
[U] \cdot [V] = \sum_X F_{U,V}^X [X]
\]
with structure constants the Hall numbers
\[
F^X_{U,V} =\#  \{ \tilde{U} \subset X  \mid  \tilde{U} \simeq U,\;\; X \slash \tilde{U} \simeq V \}.
\]
Then $\mathcal{H}_Q$ is a $\Lambda_Q^+$-graded associated algebra.

In \cite[Lemma 6.1]{reineke2003} (see also \cite[Proposition 1]{mozgovoy2013}) Reineke showed that the map
\[
\int_{\mathcal{H}}: \mathcal{H}_Q \rightarrow  \hat{\mathbb{T}}_Q, \;\;\;\;\;\;\;\;\;\;  \left[ U \right] \mapsto  \frac{q^{\frac{1}{2} \chi(\Dim\, U, \Dim\, U)}}{\# Aut(U) } x_{\Dim \, U}
\]
is a $\mathbb{Q}$-algebra homomorphism. The map $\int_{\mathcal{H}}$ is a one dimensional version of the  (partially conjectural) integration maps central to the motivic DT theory of three dimensional Calabi-Yau categories \cite{kontsevich2008}.

We want to construct a lift of the homomorphism $\int_{\mathcal{H}}$ to the self-dual setting. To do this, we first recall the definition of the Hall module $\mathcal{M}_Q$ associated to the category $Rep_{\mathbb{F}_q}(Q)$ with fixed duality structure \cite{mbyoung2012}. It is the $\mathbb{Q}$-vector space generated by symbols $[M]$ indexed by isometry classes of self-dual $\mathbb{F}_q$-representations of $Q$. The $\mathcal{H}_Q$-module structure on $\mathcal{M}_Q$ is defined by
\[
[U] \star [M] = \sum_N G^N_{U,M}  [N]
\]
with structure constants self-dual versions of Hall numbers,
\[
G^N_{U,M}= \# \{ \tilde{U} \subset N \mid \tilde{U} \simeq U, \;\; \tilde{U} \mbox{ is isotropic}, \; \; N /\!/ \tilde{U} \simeq_S M \}.
\]

The next result provides the desired lift of $\int_{\mathcal{H}}$.

\begin{Thm}
\label{thm:sdIntMap}
The map 
\[
\int_{\mathcal{M}}: \mathcal{M}_Q \rightarrow  \hat{\mathbb{S}}_Q \;\;\;\;\;\;\;\;\;\; \left[ M \right]  \mapsto \displaystyle \frac{q^{\frac{1}{2} \mathcal{E}(\Dim M)}}{\# Aut_S(M) } \xi_{\Dim M}
\]
is a $\displaystyle \int_{\mathcal{H}}$-morphism. More precisely, the diagram
\[
\begin{tikzpicture}[every node/.style={on grid},  baseline=(current bounding box.center),node distance=2]
  \node (A) {$\mathcal{H}_Q \otimes_{\mathbb{Q}} \mathcal{M}_Q$};
  \node (B) [right=of A] {$\mathcal{M}_Q$};
  \node (C) [below=of A] {$\hat{\mathbb{T}}_Q \otimes_{\mathbb{Q}(q^{\frac{1}{2}})} \hat{\mathbb{S}}_Q$};
  \node (D) [right=of C] {$\hat{\mathbb{S}}_Q$};
  \draw[->] (A)--  (B);
  \draw[->] (A)-- node [left] {$\displaystyle \int_{\mathcal{H}} \otimes \int_{\mathcal{M}}$ } (C);
  \draw[->] (B)-- node [right] {$\displaystyle \int_{\mathcal{M}}$ }(D);
  \draw[->] (C)-- (D);
\end{tikzpicture}
\]
commutes, where the horizontal maps are the module structure maps.
\end{Thm}

\begin{proof}
By linearity it suffices to show that 
\[
\int_{\mathcal{M}} ([U] \star [M])=\left( \int_{\mathcal{H}}[U] \right) \star \left(\int_{\mathcal{M}}[M] \right)
\]
for all representations $U$ and self-dual representations $M$. A direct calculation shows that this is equivalent to the identity
\[
\sum_N \frac{G^N_{U,M}}{\# Aut_S(N) } = \frac{ q^{-\chi( \Dim \, M, \Dim \, U ) - \mathcal{E}(\Dim \, U)} }{ \# Aut (U) \cdot \# Aut_S (M) }.
\]
Using \cite[Lemma 2.2]{mbyoung2012} this is in turn equivalent to the identity proven in \cite[Theorem 2.9]{mbyoung2012}.
\end{proof}

Write $\hat{\mathcal{H}}_Q$ for the completion of $\mathcal{H}_Q$ with respect to its $\Lambda_Q^+$-grading and $\hat{\mathcal{M}}_Q$ for the corresponding completion of $\mathcal{M}_Q$. Both integration maps $\int_{\mathcal{H}}$ and $\int_{\mathcal{M}}$ and Theorem \ref{thm:sdIntMap} extend to these completions.

\begin{Rem}
Theorem \ref{thm:sdIntMap} holds more generally if $Rep_{\mathbb{F}_q} (Q)$ is replaced with an exact subcategory of a hereditary finitary abelian category. The duality need only be defined on the exact subcategory. An example of this type is the category of vector bundles over a smooth projective curve over $\mathbb{F}_q$ with its standard duality.
\end{Rem}

\subsection{Orientifold DT series}

Define characteristic functions of (self-dual) representations of a fixed dimension vector $d$ in $\Lambda_Q^+$ or $\Lambda_Q^{\sigma,+}$ by
\[
\mathbf{1}_d = \sum_{\Dim \, U = d} [U] \in \mathcal{H}_Q, \;\;\;\;\;\; \;\; \mathbf{1}_d^{\sigma} = \sum_{\Dim \, M = d} [M] \in \mathcal{M}_Q.
\]
The sums run over the finitely many isomorphism (isometry) classes of (self-dual) representations of dimension vector $d$. Given a ($\sigma$-compatible) stability $\theta$, there are similarly defined characteristic functions of semistable representations with fixed dimension vector $d$ or slope $\mu$:
\[
\mathbf{1}_d^{\theta } \in \mathcal{H}_Q,  \;\;\;\;  \mathbf{1}_{\mu}^{\theta } \in \hat{\mathcal{H}}_Q, \;\;\;\;\;\;\;\; \mathbf{1}_d^{\sigma, \theta }  \in \mathcal{M}_Q, \;\;\;\; \mathbf{1}^{\sigma, \theta } \in  \hat{\mathcal{M}}_Q.
\]
As self-dual representations have zero slope we have written $\mathbf{1}^{\sigma, \theta }$ for $\mathbf{1}^{\sigma, \theta }_{\mu=0}$. Applying the appropriate integration map to each characteristic function gives a stack generating function, denoted by $A$ with the corresponding sub/superscripts. For example, 
\begin{equation}
\label{eq:ssStackGenFun}
A_d^{\sigma, \theta }= \int_{\mathcal{M}} \mathbf{1}^{\sigma, \theta }_d = \sum_{\substack{\Dim \, M = d \\ M \mbox{ is s.s}}} \frac{q^{\frac{1}{2}\mathcal{E}(d)}}{\# Aut_S(M) } \xi_d \in \hat{\mathbb{S}}_Q.
\end{equation}
In analogy with \cite{kontsevich2011} we call
\[
A^{\sigma, \theta } = \sum_{d \in \Lambda_Q^{\sigma, +}} A_d^{\sigma, \theta} \in \hat{\mathbb{S}}_Q
\]
the orientifold Donaldson-Thomas series of $(Q, \sigma)$ with its given duality structure and stability.

Let $n \geq 0$ and put $(y)_n = \prod_{i=1}^n (1-y^i)$. For each $d \in \Lambda_Q^+$ and $e \in \Lambda_Q^{\sigma, +}$ define
\[
(y)_d = \prod_{i \in Q_0} (y)_{d_i},  \;\;\;\;\;\;\;\;\;\;\;
(y)^{\sigma}_e = \prod_{i \in Q_0^{\sigma}} (y^2)_{\lfloor \frac{e_i}{2} \rfloor} \times \prod_{i \in Q_0^+} (y)_{e_i}
\]
where $\lfloor \frac{e_i}{2} \rfloor$ is the greatest integer less than or equal to $\frac{e_i}{2}$.

\begin{Prop}
\label{prop:stackGenFun}
Fix $d \in \Lambda_Q^+$, $e \in \Lambda_Q^{\sigma, +}$ and let $\theta$ be a $\sigma$-compatible stability.
\begin{enumerate}
\item The following identities hold:
\[
A_d = \frac{q^{-\frac{1}{2}\chi( d, d )}}{(q^{-1})_d} x_d, \;\;\;\;\;\; \;\;\;\;\;\;\;
A_e^{\sigma} =  \frac{q^{-\frac{1}{2}\mathcal{E}(e)}}{(q^{-1})^{\sigma}_e} \xi_e.
\]
\item The quantity $A_e^{\sigma, \theta }$ is equal to $q^{\frac{1}{2} \mathcal{E}(e)}$ times the number of $\mathbb{F}_q$-points of the stack of semistable self-dual representations of dimension vector $e$:
\[
A_e^{\sigma, \theta } = q^{\frac{1}{2} \mathcal{E}(e)} \cdot \# [R_e^{\sigma , \varepsilon, \theta \mhyphen ss } \slash G_e^{\sigma, \varepsilon}](\mathbb{F}_q) \xi_e.
\]
\end{enumerate}
\end{Prop}
\begin{proof}
The identity for $A_d$ is known  \cite{mozgovoy2013}. In the self-dual case we have
\[
A_e^{\sigma} =\sum_{\varepsilon}  q^{\frac{1}{2} \mathcal{E}(e)} \frac{\# R_e^{\sigma, \varepsilon} }{\#G_e^{\sigma, \varepsilon} } \xi_e.
\]
Denote the function on $\Lambda_Q$ given by the first two sums (last two sums) in equation \eqref{eq:modEuler} by $\mathcal{E}_0$ (respectively, $\mathcal{E}_1$). By direct inspection $\#  R_{e}^{\sigma,\varepsilon} = q^{-\mathcal{E}_1(e)}$. If $i \in Q_0^+$ then $G_{e_i}^{\sigma, \varepsilon} = GL_{e_i}$ and
\[
\frac{1}{\# G^{\sigma,\varepsilon}_{e_i} } = \frac{q^{-\mathcal{E}_0(e_i (i+ \sigma(i)))}}{(q^{-1})_{e_i}}.
\]
If $i \in Q_0^{\sigma}$ then $G_{e_i}^{\sigma, \varepsilon_i}$ is an orthogonal or symplectic group. Using the identities
\[
\# O_{2n}^{\varepsilon_i}(\mathbb{F}_q)  = \frac{2  \# GL_n (\mathbb{F}_{q^2})}{q^n+ \varepsilon_i},
\;\;\;\;\;
\# Sp_{2n}(\mathbb{F}_q)  = \frac{1}{2} \# O_{2n+1}^{\varepsilon_i}(\mathbb{F}_q) = q^n  \# GL_n (\mathbb{F}_{q^2} ),
\]
with $\varepsilon_i \in \{-1,1\}$, we find
\[
\sum_{\varepsilon_i} \frac{1}{\# G^{\sigma,\varepsilon_i}_{e_i} } = \frac{q^{-\mathcal{E}_0(e_i i)}}{(q^{-2})_{\lfloor \frac{e_i}{2} \rfloor}}.
\]
These calculations together with Burnside's lemma give the identity for $A_e^{\sigma}$.

Turning to the second part of the proposition, the number of $\mathbb{F}_q$-points of the stack $[R_e^{\sigma , \varepsilon, \theta  \mhyphen ss} \slash G_e^{\sigma, \varepsilon}]$ is by definition
\[
\# [R_e^{\sigma , \varepsilon, \theta  \mhyphen ss} \slash G_e^{\sigma, \varepsilon}](\mathbb{F}_q)= \sum_{ \eta \in Iso [R_e^{\sigma, \varepsilon, \theta  \mhyphen ss} \slash G_e^{\sigma, \varepsilon}](\mathbb{F}_q)} \frac{1}{\# Aut(\eta)}.
\]
The objects of the groupoid $[R_e^{\sigma, \varepsilon,\theta  \mhyphen ss} \slash G_e^{\sigma, \varepsilon}](\mathbb{F}_q)$ are in bijection with the set
\[
R_e^{\sigma, \theta  \mhyphen ss}(\mathbb{F}_q) \times H^1(\mathbb{F}_q, G_e^{\sigma, \varepsilon}(\overline{\mathbb{F}}_q)).
\]
The cohomology $H^1(\mathbb{F}_q, G_e^{\sigma, \varepsilon}(\overline{\mathbb{F}}_q))$ is identified with the set of inequivalent choices of $\varepsilon^{\prime}$ with $\varepsilon$ as the base point. Morphisms in the groupoid are the transporter groups
\[
Hom_{[R\slash G^{\sigma}](\mathbb{F}_q)} \left( (r^{\prime}, \varepsilon^{\prime}) , (r^{\prime \prime}, \varepsilon^{\prime\prime}) \right) = \delta_{\varepsilon^{\prime},\varepsilon^{\prime\prime}} \mbox{Trans}_{G_e^{\sigma, \varepsilon^{\prime}}(\mathbb{F}_q)}(r^{\prime}, r^{\prime \prime}).
\]
Hence the automorphisms of $(r^{\prime}, \varepsilon^{\prime})$ are the stabilizers of $r^{\prime} \in R_e^{\sigma, \varepsilon^{\prime},\theta  \mhyphen ss}(\mathbb{F}_q)$ under the action of $G_e^{\sigma, \varepsilon^{\prime}}(\mathbb{F}_q)$, or in other words, $Aut_S(r^{\prime})$. The proposition follows after using equation \eqref{eq:ssStackGenFun}.
\end{proof}

Our next goal is to describe the characteristic function $\mathbf{1}_d^{\sigma, \theta }$ for a given $\sigma$-compatible stability $\theta$. As iterated products in the Hall algebra count filtrations of representations, the existence of unique HN filtrations implies the following identity in $\mathcal{H}_Q$ (see \cite{reineke2003}):
\begin{equation}
\label{eq:HNRec}
\mathbf{1}_d = \sum_{d^{\bullet}} \mathbf{1}_{d^1}^{\theta } \cdots \mathbf{1}_{d^n}^{\theta }.
\end{equation}
The sum is over all $n \geq 1$ and $d^{\bullet}=(d^1, \dots, d^n) \in (\Lambda_Q^+)^n$ of weight $d=\sum_{i=1}^n d^i$ whose slopes are strictly decreasing, $\mu(d^1) > \cdots > \mu (d^n)$. Equation \eqref{eq:HNRec} gives a recursion for $\mathbf{1}_d^{\theta }$ in terms of $\mathbf{1}_{d^{\prime}}$ with $\dim d^{\prime} \leq \dim d$. This recursion was solved by Reineke \cite[Theorem 5.1]{reineke2003}.

Using Proposition \ref{prop:sdHNFilt}, similar reasoning gives an identity in $\mathcal{M}_Q$: 
\begin{equation}
\label{eq:sigmaHNRec}
\mathbf{1}^{\sigma}_d = \sum_{(d^{\bullet}; d^{\infty})} \mathbf{1}_{d^1}^{\theta } \cdots \mathbf{1}_{d^n}^{\theta } \star \mathbf{1}^{\sigma,\theta }_{d^{\infty}}.
\end{equation}
The sum is now over all $n \geq 0$ and $(d^{\bullet}; d^{\infty})=(d^1, \dots, d^n; d^{\infty}) \in (\Lambda_Q^+)^n \times \Lambda_Q^{\sigma,+}$ of $\sigma$-weight $d=\sum_{i=1}^n H(d^i) + d^{\infty}$ whose slopes are strictly decreasing. Note that $d^{\infty}$ may be zero but that $d^i \neq 0$ if $i \neq 0$. We write $l(d^{\bullet}) = n$ if $d^{\bullet} \in (\Lambda_Q^+)^n$.

\begin{Def}[\textit{cf.} {\cite[Definition 5.2]{reineke2003}}]
Let $(d^{\bullet}; d^{\infty}) \in (\Lambda_Q^+)^n \times \Lambda_Q^{\sigma, +}$.
\begin{enumerate}[leftmargin=0cm,itemindent=.6cm,labelwidth=\itemindent,labelsep=0cm,align=left]
\item  For a possibly empty subset $I= \{ s_1 < \cdots < s_k \} \subset \{1, \dots, n\}$, the $I$-coarsening of $(d^{\bullet}; d^{\infty})$ is
\[
c_I(d^{\bullet}; d^{\infty})= (d^1+ \cdots + d^{s_1},  \cdots, d^{s_{k-1}+1} + \cdots + d^{s_k}; H(d^{s_k+1} + \cdots + d^n) + d^{\infty}).
\]

\item The $I$-coarsening $c_I(d^{\bullet}; d^{\infty})$ is called $\sigma$-admissible if
\begin{enumerate}[leftmargin=0.5cm,itemindent=.6cm,labelwidth=\itemindent,labelsep=0cm,align=left]
\item its components have strictly decreasing slope,
\item for each $i =1, \dots, k$ and $j^{\prime} =s_{i-1}+1, \dots, s_i-1$ the inequality 
\[
\mu(\sum_{j=s_{i-1} + 1}^{j^{\prime}} d^j) > \mu(\sum_{j=s_{i-1} + 1}^{s_i} d^j)
\]
holds, and
\item for each $j^{\prime} =s_k+1, \dots, n-1$ the inequality $\mu(\sum_{j=s_k + 1}^{j^{\prime}} d^j) > 0$ holds.
\end{enumerate}
\end{enumerate}
\end{Def}

We now solve the recursion \eqref{eq:sigmaHNRec} for $\mathbf{1}_d^{\sigma,\theta }$.

\begin{Thm}
\label{thm:recReso}
For each $d \in \Lambda_Q^{\sigma, +}$, equation \eqref{eq:sigmaHNRec} is solved by
\[
\mathbf{1}_d^{\sigma, \theta } = \sum_{(d^{\bullet}; d^{\infty})} (-1)^n  \mathbf{1}_{d^1} \cdots \mathbf{1}_{d^n} \star \mathbf{1}^{\sigma}_{d^{\infty}}
\]
where the sum is over all $n \geq 0$ and $(d^{\bullet}; d^{\infty}) \in (\Lambda_Q^+)^n \times \Lambda_Q^{\sigma, +}$ which are equal to $(\varnothing; d)$ or satisfy $\mu (\sum_{i=1}^j d^i) >0$ for $j=1, \dots, n$ and have $\sigma$-weight $d$.
\end{Thm}

\begin{proof}
Using the resolution of the HN recursion \eqref{eq:HNRec} from \cite{reineke2003} and substituting the claimed expression for $\mathbf{1}_d^{\sigma,\theta }$ into equation \eqref{eq:sigmaHNRec} gives for $\mathbf{1}_d^{\sigma}$ the expression
\begin{equation}
\label{eq:recurExpress}\sum_{(d^{\bullet}; d^{\infty})}  \sum_{(d^{1, \bullet}, \cdots , d^{n, \bullet} ;  d^{\infty, \bullet})} (-1)^{ \sum_{i=1}^n (l_i -1) + l_{\infty}}  \left( \prod_{i=1}^{\substack{\longrightarrow \\ \infty}} \prod_{j=1}^{\substack{\longrightarrow \\ l_i}} \mathbf{1}_{d^{i,j}} \right) \star \mathbf{1}_{d^{\infty, \infty}}^{\sigma}.
\end{equation}
The outer sum is as in equation \eqref{eq:sigmaHNRec} while the inner sum is over all $(d^{1, \bullet}, \cdots , d^{n, \bullet} ;  d^{\infty, \bullet})$ with $d^{k, \bullet} \in (\Lambda_Q^+)^{l_k}$ of weight $d^k$ satisfying
\[
\mu (\sum_{i=1}^l d^{k, i}) > \mu(d^k), \;\;\;\;\;\; l = 1, \dots, l_k -1
\]
and $d^{\infty, \bullet} \in (\Lambda_Q^+)^{l_{\infty}} \times \Lambda_Q^{\sigma,+}$ of $\sigma$-weight $d^{\infty}$ satisfying
\[
\mu (\sum_{i=1}^{l} d^{\infty,i}) >0, \;\;\;\;\;\; l=1, \dots, l_{\infty}.
\]

Let $(e^{\bullet}; e^{\infty})$ be the concatenation of $d^{1, \bullet}, \dots, d^{n,\bullet}$ and $d^{\infty, \bullet}$,
\[
(e^{\bullet}; e^{\infty})= \left(  d^{1,1}, \dots, d^{n, l_n}, d^{\infty, 1}, \dots, d^{\infty, l_{\infty}} ; d^{\infty, \infty} \right).
\]
Then $(d^{\bullet}; d^{\infty})$ is a $\sigma$-admissible coarsening of $(e^{\bullet}; e^{\infty})$. Since
\[
\sum_{i=1}^n (l_i -1) + l_{\infty} = l(e^{\bullet}) - l( d^{\bullet})
\]
the order of summation in \eqref{eq:recurExpress} can be swapped to give
\[
\mathbf{1}^{\sigma}_d = \sum_{(e^{\bullet}; e^{\infty})} (-1)^{l(e^{\bullet}) }  \sum_{(d^{\bullet}, d^{\infty})} (-1)^{l(d^{\bullet})} \mathbf{1}_{e^1} \cdots \mathbf{1}_{e^{l(e^{\bullet})}}  \star \mathbf{1}_{e^{\infty}}^{\sigma}.
\]
The range of the outer sum is as in the statement of the theorem while the inner sum is over all $\sigma$-admissible coarsenings of $(e^{\bullet}; e^{\infty})$.

To complete the proof it suffices to show that for fixed $(e^{\bullet}; e^{\infty})$ equal to $(\varnothing; d)$ or satisfying the inequality in the statement of the theorem we have
\begin{equation}
\label{eq:imptIden}
\sum_{(d^{\bullet}; d^{\infty})} (-1)^{l(d^{\bullet})} = \left\{ \begin{array}{cl}  1, & \mbox{ if } (e^{\bullet}; e^{\infty}) = (\varnothing; d)  \\  0, & \mbox{ otherwise} \\\end{array} \right. ,
\end{equation}
the sum being over all $\sigma$-admissible coarsenings of $(e^{\bullet}; e^{\infty})$. This is a self-dual analogue of \cite[Lemma 5.4]{reineke2003}. If $l(e^{\bullet}) =0$, then $(e^{\bullet}; e^{\infty}) = (\varnothing; d)$ and equation \eqref{eq:imptIden} is trivially true. For $l(e^{\bullet}) \geq 1$ we proceed by induction. If $l(e^{\bullet}) =1$, then $(e^{\bullet}; e^{\infty}) = (e^1; e)$ with $\mu(e^1) > 0$. This has $\sigma$-admissible coarsenings $I=\varnothing$ and $I = \{ 1\}$ and equation \eqref{eq:imptIden} again holds. For $l(e^{\bullet}) \geq 2$ we can follow the proof of \cite[Lemma 5.4]{reineke2003}, distinguishing the cases $\mu(e^1) < \mu (e^2)$ and $\mu(e^1) \geq \mu(e^2)$. This allows to complete the induction step.
\end{proof}

For $d^{\bullet} \in (\Lambda_Q^+)^n $ and $e \in \Lambda_Q^{\sigma, +}$ introduce the notation
\[
\chi( d^{\bullet} ) = \sum_{1 \leq i < j \leq n} \chi( d^j , d^i ), \;\;\;\;\chi( e, d^{\bullet} ) = \sum_{i=1}^n \chi ( e, d^i ), \;\;\;\;\mathcal{E}(d^{\bullet}) = \mathcal{E}( \sum_{i =1 }^n d^i ).
\]
By applying the Hall module integration map to the expression for $\mathbf{1}_d^{\sigma, \theta }$ from Theorem \ref{thm:recReso} and then using Theorem \ref{thm:sdIntMap} and Proposition \ref{prop:stackGenFun} we obtain the following result.

\begin{Thm}
\label{thm:countInv}
For any $d \in \Lambda_Q^{\sigma,+}$ and $\sigma$-compatible stability $\theta$, the coefficient of $\xi_d$ in $A_d^{\sigma,\theta }$ is equal to
\[
q^{\frac{1}{2} \mathcal{E}(d)} \sum_{(d^{\bullet}; d^{\infty})} (-1)^{l(d^{\bullet})} q^{-\chi( d^{\bullet} ) - \chi( d^{\infty}, d^{\bullet} ) - \mathcal{E}(d^{\bullet})}  \left( \prod_{i=1}^{l(d^{\bullet})}  \frac{q^{-\chi( d^i , d^i )}}{(q^{-1})_{d^i}}  \right)  \frac{q^{-\mathcal{E}(d^{\infty})}}{(q^{-1})^{\sigma}_{d^{\infty}}},
\]
where the range of summation is as in Theorem \ref{thm:recReso}.

In particular, there exists a rational function $a_d^{\sigma,\theta } (t) \in \mathbb{Z} (t^{\frac{1}{2}})$ that specializes to $A_d^{\sigma, \theta }(\mathbb{F}_q)$ at every odd prime power $q$.
\end{Thm}

For an acyclic quiver $Q$ and a sufficiently generic stability $\theta$, the ordinary moduli space $\mathfrak{M}_d^{\theta }$, over $k=\mathbb{C}$ say, is a smooth projective variety. In  \cite{reineke2003} the Weil conjectures are used to show that the function $a_d^{\theta }$, specializing to $A_d^{\theta}$ at each prime power, satisfies
\[
a_d^{\theta }(v^2) = v^{\chi(d,d)} (v^2-1)^{-1}  P_{\mathfrak{M}_d^{\theta }}(v),
\]
giving an effective way to the compute the Poincar\'{e} polynomial $P_{\mathfrak{M}_d^{\theta }}(v)$. In the self-dual case, the requirement that $\theta$ be $\sigma$-compatible means that it cannot be chosen generically, except in some low dimensional examples. This leads to the existence of strictly semistable self-dual representations so that $a_d^{\sigma,\theta }$ is not obviously related to the Poincar\'{e} polynomial of $\mathfrak{M}_d^{\sigma,\theta }$. Instead, $a_d^{\sigma,\theta }$ can be interpreted as the Poincar\'{e} series of the moduli stack $[R_d^{\sigma,\theta } \slash G_d^{\sigma}]$. For similar interpretations in the case of $G$-bundles over curves and ordinary quiver representations see \cite{atiyah1983}, \cite{laumon1996} and \cite{harada2011} respectively.

The functions $a_d^{\theta }$ also have string theoretic importance, regardless of whether $\theta$ is generic or not. In \cite{manschot2013} it was proposed that the functions $a_d^{\theta }$ determine the Higgs branch expression for the index of multi-centred BPS black holes in $\mathcal{N}=2$ supergravity. Using the explicit computation of $a_d^{\theta }$ from \cite{reineke2003} this proposal was tested in a number of examples. It would be interesting to test a similar relationship between $a_d^{\sigma,\theta }$ and a Coulomb branch formula for indices of BPS black holes in the presence of an orientifold \cite{denef2010}.

\begin{Ex}
For the $n$-Kronecker quiver with stability $\theta_i=-i$ we have
\[
a_{(1,1)}^{\mathfrak{sp}, \theta } (t) =  t^{\frac{1}{2}\mathcal{E}(1,1)} \frac{t^n-1}{t-1} = t^{\frac{1-n}{2}} [n]_t.
\]
Indeed, there are $2 [n]_q$ isometry classes of semistable symplectic $\mathbb{F}_q$-representations of dimension vector $(1,1)$, each having isometry group $\mathbb{Z}_2$. These representations are absolutely $\sigma$-stable. The moduli space is $\mathfrak{M}^{\mathfrak{sp}, \theta } _{(1,1)} \simeq \mathbb{P}^{n-1}$ and we recover its Poincar\'{e} polynomial via
\[
a_{(1,1)}^{\mathfrak{sp}, \theta }(v^2) = v^{\mathcal{E}(1,1)} P_{\mathfrak{M}^{\mathfrak{sp}, \theta } _{(1,1)}}(v).
\]

In general $a_{(d,d)}^{\mathfrak{sp}, \theta }$ is rational, even after multiplication by $t^{-\frac{1}{2} \mathcal{E}(d,d)}$. For example, using Theorem \ref{thm:countInv} we compute
\[
a_{(2,2)}^{\mathfrak{sp}, \theta }(t)= t^{\frac{1}{2} \mathcal{E}(2,2)} \frac{t^{n-1}[2n]_t - [n]_t}{t+1}.
\]
When $n=2$ this is $a_{(2,2)}^{\mathfrak{sp}, \theta } (t) = t^{\frac{1}{2} \mathcal{E}(2,2)}(t^3 + t -1)$ while after multiplication by $t^{-\frac{1}{2} \mathcal{E}(2,2)}$ is polynomial but fails to recover the Poincar\'{e} polynomial of $\mathfrak{M}_{(2,2)}^{\mathfrak{sp}, \theta } \simeq \mathbb{P}^2$ because of strictly semistable symplectic representations.
\end{Ex}

We now describe a class of quivers with involution whose ordinary and self-dual representation theories differ rather mildly. This is partially motivated by \cite[\S 5.2.1]{denef2011}. Let $Q$ be an acyclic quiver and let $Q^{\sqcup}$ be the disjoint union of $Q$ with its opposite $Q^{op}$. Then $Q^{\sqcup}$ has a canonical involution $\sigma$ that swaps the nodes and arrows of $Q$ and $Q^{op}$. Let $Q^{\prime}$ be a quiver obtained from $Q^{\sqcup}$ by adjoining arrows from $Q^{op}$ to $Q$ in such a way that $\sigma$ can be extended to $Q^{\prime}$. The $\sigma$-compatible stabilities of $Q^{\prime}$ are of the form $\theta^{\prime} = \theta - \sigma^*\theta$ with $\theta \in \Lambda_Q^{\vee}$. Given $d \in \Lambda_Q^+$ pick stabilities $\theta_0, \theta_ - \in \Lambda_Q^{\vee}$ satisfying $\theta_0(d) =0$ and $\theta_-(d) < 0$. Assume that $\theta_0$ and $\theta_-$ are generic in the sense that all semistable representations of $Q$ of dimension vector $d$ are stable.

Fix a duality structure on $Q^{\prime}$. Any self-dual representation of dimension vector $d^{\prime}=H(d)$ can be written uniquely as a Lagrangian extension
\begin{equation}
\label{eq:lagFilt}
0 \rightarrow U \rightarrow N \rightarrow S(U) \rightarrow 0
\end{equation}
with $U$ a representation of $Q$ of dimension vector $d$.

The representation $N$ is $\theta^{\prime}_0$-semistable if and only if $U$ is $\theta_0$-semistable. In this case the $\sigma$-Jordan-H\"{o}lder filtration of $N$ coincides with the Jordan-H\"{o}lder filtration of $U$. This implies that the map
\[
\mathfrak{M}^{\sigma, \theta^{\prime}_0 }_{d^{\prime}}(Q^{\prime})  \rightarrow \mathfrak{M}^{\theta_0 }_d (Q), \;\;\;\;\; N \mapsto U
\]
is an isomorphism. It is straightforward to verify that the Lagrangian extensions \eqref{eq:lagFilt} are parameterized by the vector space $Ext^1(S(U),U)^S$. Since $Hom(S(U),U)$ is trivial, we have $\dim Ext^1(S(U),U)^S =-\mathcal{E}(d)$. Using this we compute
\[
a_{d^{\prime}}^{\sigma, \theta_0^{\prime} } (v^2)= v^{\mathcal{E}(d^{\prime})}\frac{v^{- 2\mathcal{E}(d)}}{v^2-1} P_{\mathfrak{M}^{\sigma, \theta^{\prime}_0 }_{d^{\prime}}(Q^{\prime})}(v).
\]

On the other hand, in some examples $\mathfrak{M}_{d^{\prime}}^{\sigma, \theta^{\prime}_-  }(Q^{\prime})$ is a fibration over $\mathfrak{M}^{\sigma, \theta^{\prime}_0 }_{d^{\prime}}(Q^{\prime})$ with fibres weighted projective spaces of dimension $-\mathcal{E}(d) -1$.

\begin{Ex}
As an example of the previous discussion, let $Q$ be the $n$-Kronecker quiver on nodes $\{-2,-1\}$ and let $Q^{\prime}$ be the quiver
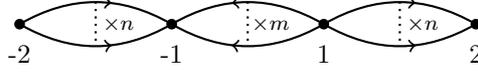
\begin{center}
\begin{tikzpicture}[thick,scale=1,decoration={markings,mark=at position 0.6 with {\arrow{>}}}]
\draw[postaction={decorate}] (0,0) [out=30,in=150]  to  (2,0);
\draw[postaction={decorate}] (0,0) [out=-30,in=210]  to  (2,0);
\draw[postaction={decorate}] (4,0) [out=150,in=30]  to  (2,0);
\draw[postaction={decorate}] (4,0) [out=210,in=-30]  to  (2,0);
\draw[postaction={decorate}] (4,0) [out=30,in=150]  to  (6,0);
\draw[postaction={decorate}] (4,0) [out=-30,in=210]  to  (6,0);

\draw[dotted] (1,.3) to (1,-.3);
\draw[dotted] (3,.3) to (3,-.3);
\draw[dotted] (5,.3) to (5,-.3);

\draw (0,-.4) node {-2};
\draw (2,-.4) node {-1};
\draw (4,-.4) node {1};
\draw (6,-.4) node {2};

\draw (1.3,0) node {\footnotesize $ \times n$};
\draw (3.3,0) node {\footnotesize $ \times m$};
\draw (5.3,0) node {\footnotesize $ \times n$};

\fill (0,0) circle (2pt);
\fill (2,0) circle (2pt);
\fill (4,0) circle (2pt);
\fill (6,0) circle (2pt);
\end{tikzpicture}
\end{center}
A symplectic representation of $Q^{\prime}$ of dimension vector $d^{\prime}=(d_2, d_1, d_1, d_2)$ is a tuple
\[
(A,B)  \in Hom( k^{d_2}, k^{d_1})^{\oplus n} \oplus (\mbox{Sym}^2 k^{d_1})^{\oplus m}.
\]

Suppose that $d_1=1$. For stability $\theta_0=( d_2 ,-1)$ the representation $(A,B)$ is semistable if and only if $A \neq 0$. From the discussion above
\[
\mathfrak{M}_{d^{\prime}}^{\mathfrak{sp}, \theta^{\prime}_0 }(Q^{\prime}) \simeq \mathfrak{M}_{(d_2,1)}^{\theta_0 }(Q) \simeq Gr(d_2, \mathbb{C}^n)
\]
and
\[
a^{\mathfrak{sp}, \theta^{\prime}_0 }_{d^{\prime}}(v^2) = v^{\mathcal{E}(d^{\prime})} \frac{v^{2m} }{v^2-1}\left[ \begin{matrix} n \\ d_2 \end{matrix} \right]_v.
\]
For $\theta_-=(d_2+1, -1)$ the representation $(A,B)$ is semistable if and only if neither $A$ nor $B$ is zero. Then $\mathfrak{M}_{d^{\prime}}^{\mathfrak{sp}, \theta^{\prime}_-  }(Q^{\prime})$ is a $\mathbb{P}^{m-1}$-fibration  over $\mathfrak{M}_{d^{\prime}}^{\mathfrak{sp}, \theta^{\prime}_0 }(Q^{\prime})$ and
\[
a^{\mathfrak{sp}, \theta_-^{\prime} }_{d^{\prime}} (v^2)= v^{\mathcal{E}(d^{\prime})} [m]_v \left[ \begin{matrix} n \\ d_2 \end{matrix} \right]_v = v^{\mathcal{E}(d^{\prime})} P_{\mathfrak{M}_{d^{\prime}}^{\mathfrak{sp}, \theta^{\prime}_-  }(Q^{\prime})}(v).
\]
Note that the computation of $a^{\mathfrak{sp}, \theta_-^{\prime} }_{d^{\prime}}$ takes into account the non-trivial $\mathbb{Z}_2$-gerbe structure of the fibres.
\end{Ex}

\subsection{Wall-crossing of orientifold DT invariants}

We begin this section by describing the expected wall-crossing behaviour of counts of $\sigma$-stable self-dual objects in $Rep_{\mathbb{C}}(Q)$. Na\"{i}vely, these numbers are the orientifold DT invariants. A more precise approach to wall-crossing is described below.

If $\theta$ is generic in the sense that all semistable representations of dimension vector $d$ are stable, then $\mathfrak{M}_d^{\theta }$ is smooth and the numerical DT invariant is the topological Euler characteristic
\[
\Omega^{\theta}_d= (-1)^{\dim  \mathfrak{M}^{\theta }_d} \chi(\mathfrak{M}_d^{\theta }) \in \mathbb{Z}.
\]
The definition of $\Omega^{\theta}_d$ for general $\theta$ and $d$ is more involved; see \cite{joyce2012}, \cite{kontsevich2008}, \cite{kontsevich2011} and equation \eqref{eq:dtFactorization} below. Under similar generic conditions we define the numerical orientifold DT invariant by
\[
\Omega_e^{\sigma, \theta} = (-1)^{\dim  \mathfrak{M}^{\sigma, \theta }_e} \chi(\mathfrak{M}_e^{\sigma, \theta }).
\]
By convention we set $\Omega_0^{\sigma, \theta} =1$ for all $\theta$.

To study the $\theta$ dependence of $\Omega^{\sigma, \theta}$ fix an object $U$ and a self-dual object $M$ with $\Dim\, U = d$ and $\Dim \, M = e$. Assume that $d, \sigma(d)$ and $e$ are distinct and primitive. Let $\theta_-$, $ \theta_0$ and $\theta_+$ be nearby $\sigma$-compatible stabilities such that $U$ and $M$ are stable with respect to all three and
\[
\mu_{\theta_-} (U) <0, \;\;\;\;\; \mu_{\theta_0}(U) =0 , \;\;\;\;\; \mu_{\theta_+}(U) > 0.
\]

For stability $\theta_-$ we can obtain new $\sigma$-stable self-dual representations from $U$ and $M$ through non-trivial self-dual extensions of the form
\[
0 \rightarrow U \rightarrow N \dashrightarrow M \rightarrow 0,
\]
presenting $M$ as a quotient of $U^{\perp} \subset N$ by $U$. On the other hand, for stability $\theta_+$ the representation $N$ is destabilized by $U$ whereas non-trivial self-dual extensions of $M$ by $S(U)$ may now be $\sigma$-stable. Since $U$, $S(U)$ and $M$ are stable and pairwise non-isomorphic, Schur's lemma implies that there are no non-zero morphisms between them. In this case the self-dual extensions can be decomposed as
\[
Ext^1_{s.d.}(M,U) \simeq Ext^1 (M,U) \times Ext^1(S(U),U)^S
\]
and
\[
Ext^1_{s.d.}(M,S(U)) \simeq Ext^1 (M,S(U)) \times Ext^1(U,S(U))^S.
\]
Roughly,  the first factor in $Ext^1_{s.d.}(M,U)$ describes the extension class of $U^{\perp}$ while the second factor describes the self-dual representation $N$ as an extension of $S(U)$ by $U^{\perp}$. See \cite[\S 2.3]{mbyoung2012} for details.

If indeed all non-trivial self-dual extensions described above are $\sigma$-stable, in passing from stability $\theta_-$ to $\theta_+$ we therefore expect to gain\footnote{These are weighted projective spaces: the natural action of $\mathbb{C}^{\times}$ on $Ext^1_{s.d.}(M,S(U))$ has weight one on $Ext^1 (M,S(U))$ and weight two on $Ext^1(U,S(U))^S$. Since the topological Euler characteristics is not affected by the weighting it suffices to think of them as ordinary projective spaces.} $\mathbb{P} Ext^1_{s.d.}(M,S(U))$ and lose $\mathbb{P} Ext_{s.d.}^1(M,U)$ worth of $\sigma$-stable representations of dimension vector $H(d) + e$, leading to a change in $\chi(\mathfrak{M}_{H(d)+e}^{\sigma})$ of
\[
\chi (\mathbb{P} Ext_{s.d.}^1(M,S(U)) ) - \chi(\mathbb{P} Ext_{s.d.}^1(M,U)) = 
\langle M, U \rangle + \tilde{\mathcal{E}}(U).
\]
Let $\mathcal{I}: \Lambda_Q \times \Lambda_Q^{\sigma} \rightarrow \mathbb{Z}$ be the function defined by the expression on the right hand side of this equation. As $U$ and $M$ vary over their respective moduli spaces the total change in $\Omega_{H(d) + e}^{\sigma}$ is
\begin{equation}
\label{eq:oriWCF}
\Delta \Omega^{\sigma, \theta_- \rightarrow \theta_+}_{H(d) + e} = (-1)^{\mathcal{I}(d,e) -1}  \mathcal{I}(d,e)   \Omega^{\theta_0}_d \Omega^{\sigma, \theta_0}_e.
\end{equation}
Note that this equation is already non-trivial in the Lagrangian case, where $e =0$. In this specialization the above argument can be made more precise by a slight modification of \cite[\S 4.3]{stoppa2011}.

Equation \eqref{eq:oriWCF} is an orientifold modification of the primitive wall-crossing formula for BPS indices in four dimensional theories with $\mathcal{N}=2$ supersymmetry \cite{denef2011}. A physical derivation of equation \eqref{eq:oriWCF} was given in the setting of four dimensional $\mathcal{N}=2$ supergravity in an orientifold background \cite{denef2010}. Physically, the function $\mathcal{I}$ is a parity twisted Witten index \cite{brunner2004}, counting orientifold invariant open string states between an arbitrary $D$-brane configuration of charge $d \in \Lambda_Q$, its orientifold image and an orientifold invariant $D$-brane configuration of charge $e \in \Lambda_Q^{\sigma}$. The charge of the orientifold plane is implicit in $\mathcal{I}$.

\begin{Ex}
Consider again the quiver $Q^{\prime}$, a modification of $Q^{\sqcup}$. When $\mathfrak{M}_{d^{\prime}}^{\sigma , \theta^{\prime}_- }(Q^{\prime})$ is indeed a weighted $\mathbb{P}^{-\mathcal{E}(d)-1}$-bundle over $\mathfrak{M}^{\theta_0 }_d(Q )$ we have
\[
\chi( \mathfrak{M}_{d^{\prime}}^{\sigma, \theta^{\prime}_- }(Q^{\prime}) )= - \mathcal{E}(d)  \cdot  \chi ( \mathfrak{M}^{\theta_0 }_d(Q )).
\]
As $\mathfrak{M}_{d^{\prime}}^{\sigma, \theta^{\prime}_+ }(Q^{\prime})$ is empty (the representation $U$ in \eqref{eq:lagFilt} destabilizes $N$), this agrees with the Lagrangian specialization of equation \eqref{eq:oriWCF}.
\end{Ex}

We now return to the finite field setting and discuss a more rigorous approach to wall-crossing. We first prove a general wall-crossing formula for orientifold DT series and then specialize it to the finite type case where it becomes much more explicit and can be compared with equation \eqref{eq:oriWCF}.

\begin{Thm}
\label{thm:sdWallCross}
For any two $\sigma$-compatible stabilities $\theta$ and $\theta^{\prime}$, the identity
\[
\overset{\longleftarrow}{\prod_{\mu\in \mathbb{Q}_{> 0} }} A_{\mu}^{\theta } \star A^{\sigma, \theta } = \overset{\longleftarrow}{\prod_{\mu \in \mathbb{Q}_{> 0} }} A_{\mu}^{\theta^{\prime} } \star A^{\sigma, \theta^{\prime} }
\]
holds in $\hat{\mathbb{S}}_Q$.
\end{Thm}
\begin{proof}
Writing equation \eqref{eq:sigmaHNRec} in terms of semistable characteristic functions with fixed slope gives the following equalities in $\hat{\mathcal{M}}_Q$:
\begin{equation}
\label{eq:HNinHallMod}
\overset{\longleftarrow}{\prod_{\mu \in \mathbb{Q}_{> 0}}} \mathbf{1}_{\mu}^{\theta } \star \mathbf{1}^{\sigma, \theta }= 
\mathbf{1}^{\sigma} = \overset{\longleftarrow}{\prod_{\mu \in \mathbb{Q}_{> 0}}} \mathbf{1}_{\mu}^{\theta^{\prime} } \star \mathbf{1}^{\sigma, \theta^{\prime} }.
\end{equation}
Applying $\int_{\mathcal{M}}$ and using Theorem \ref{thm:sdIntMap} gives the desired identity.
\end{proof}

The quantum dilogarithm is the series
\[
\mathbb{E}_q(x) = \sum_{n \geq 0} \frac{q^{\frac{n^2}{2}}}{(q^n-1) \cdots (q^n - q^{n-1})} x^n \in \mathbb{Q}(q^{\frac{1}{2}}) \pser{ x }.
\]
We recall one way that $\mathbb{E}_q(x)$ arises in DT theory. Fix a $\sigma$-compatible stability $\theta$ and a rigid absolutely stable representation $U$ with $\Dim\, U = d$. The subcategory of $Rep_{\mathbb{F}_q}(Q)$ generated by $U$ consists of semistable representations $\{U^{\oplus n}\}_{n \geq 1}$ with $Aut(U^{\oplus n}) \simeq GL_n(\mathbb{F}_q)$. The contribution to $\int_{\mathcal{H}} \mathbf{1}$ generated by $U$ is then
\[
A_U^{\theta } = \int_{\mathcal{H}} \sum_{n=0}^{\infty} [U^{\oplus n}] = \sum_{n=0}^{\infty} \frac{q^{\frac{n^2}{2}}}{\# GL_n (\mathbb{F}_q)} x_{nd} =\mathbb{E}_q(x_d).
\]

Suppose now that $U \simeq S(U)$. Under the $\mathbb{Z}_2$-action determined by $(S, \Theta)$, $End(U)$ is either the sign or trivial representation. In the former case a self-dual structure on $U^{\oplus n}$ is a non-singular skew-symmetric element in $End(U^{\oplus n})$. The self-dual representations generated by $U$ are $\{H(U^{\oplus n}) \}_{n \geq 1}$. These are semistable with $Aut_S(H(U^{\oplus n})) \simeq  Sp_{2n}(\mathbb{F}_q)$ so that the contribution to $\int_{\mathcal{M}} \mathbf{1}^{\sigma}$ generated by $U$ is
\[
A_U^{\mathfrak{sp}, \theta  } = \int_{\mathcal{M}} \sum_{n=0}^{\infty} [H(U^{\oplus n})] \\
= \sum_{n=0}^{\infty} \frac{q^{\frac{n(2n+1)}{2}}}{q^n \# GL_n (\mathbb{F}_{q^2})} \xi_{2n d} \\
= \mathbb{E}_{q^2}(q^{-\frac{1}{2}} x_d) \star \xi_0.
\]
When $End(U)$ is the trivial representation, a self-dual structure on $U^{\oplus n}$ is a non-singular symmetric element in $End(U^{\oplus n})$ so that the self-dual representations generated by $U$ are $\{U^{\oplus n, \varepsilon} \}_{n \geq 1, \varepsilon \in \{\pm\}}$, each being semistable with $Aut_S(U^{\oplus n, \varepsilon} ) \simeq O_n^{\varepsilon}(\mathbb{F}_q)$. The contribution is
\begin{eqnarray*}
A_U^{\mathfrak{o}, \theta  } &=&  \int_{\mathcal{M}} \sum_{n =0}^{\infty} \sum_{\varepsilon \in \{\pm \}} [U^{\oplus n, \varepsilon}] \\
&=& \sum_{n=0}^{\infty}  \frac{q^{\frac{n(2n-1)}{2}}q^{n}}{\# GL_n(\mathbb{F}_{q^2})} \xi_{2n d}  + \sum_{n=0}^{\infty}  \frac{q^{\frac{n(2n+1)}{2}}q^{-n}}{\# GL_n(\mathbb{F}_{q^2})} \xi_{(2n+1) d} \\
&=& \mathbb{E}_{q^2}(q^{\frac{1}{2}} x_d) \star \xi_0 + \mathbb{E}_{q^2}(q^{-\frac{1}{2}} x_d) \star \xi_d.
\end{eqnarray*}

In these calculations the factors $\mathbb{E}_{q^2}(q^{\bullet} x_d)$ represent the contributions of the hyperbolics $H(U^{\oplus n})$ to $\int_{\mathcal{M}} \mathbf{1}^{\sigma}$. The simple form of $A_U^{\mathfrak{sp}, \theta }$ reflects that all self-dual representations generated by $U$ are hyperbolic. In particular, there are no $\sigma$-stable representations. The form of $A_U^{\mathfrak{o}, \theta }$ is more interesting, consisting of two terms. The first includes contributions from the hyperbolics $H(U^{\oplus n})$ and the non-split representations $H(U^{\oplus n}) \oplus U^{\oplus 2, -}$. Note that over $\overline{\mathbb{F}}_q$ the latter are also hyperbolic. The second term consists of contributions from $H(U^{\oplus n}) \oplus U^{\varepsilon}$. It is this term that contains information about the absolutely $\sigma$-stable representations.

We now turn to the simplest case of Theorem \ref{thm:sdWallCross}.

\begin{Ex}
Let $Q$ be the $A_2$ Dynkin quiver
\[
\begin{tikzpicture}[thick,scale=.8,decoration={markings,mark=at position 0.6 with {\arrow{>}}}]
\draw[postaction={decorate}] (0,0)  to  (2,0);

\draw (0,-.4) node {\small $-1$};
\draw (2,-.4) node {\small $1$};

\fill (0,0) circle (2pt);
\fill (2,0) circle (2pt);
\end{tikzpicture}
\vspace{-5pt}
\]
The wall-crossing formula for ordinary quiver representations is the pentagon identity in $\hat{\mathbb{T}}_Q$:
\[
\mathbb{E}_q(x_1)  \cdot \mathbb{E}_q(x_{-1})= \mathbb{E}_q(x_{-1})  \cdot  \mathbb{E}_q(x_{(1,1)})  \cdot   \mathbb{E}_q(x_1).
\]
It is the simplest instance of the primitive wall-crossing formula for DT invariants \cite{dimofte2011}, \cite{kontsevich2011}.
The stabilities are $\theta_i=-i$ and $\theta^{\prime}_i=i$ on the left and right hand side of this equation respectively.

For orthogonal representations Theorem \ref{thm:sdWallCross} gives the $\hat{\mathbb{S}}_Q$-identity
\[
\mathbb{E}_q(x_1) \star \xi_0 = \mathbb{E}_q(x_{-1}) \star A^{\mathfrak{o},ss},
\]
the stabilities as above. The factor $A^{\mathfrak{o},ss}$ is generated by the non-simple indecomposable. As this representation does not admit an orthogonal structure we have $A^{\mathfrak{o},ss} = \mathbb{E}_{q^2}(q^{-\frac{1}{2}} x_{(1,1)}) \star \xi_0$. The wall-crossing formula for orthogonal representations therefore reads
\[
\mathbb{E}_q(x_1) \star \xi_0 = \mathbb{E}_q(x_{-1})  \cdot  \mathbb{E}_{q^2}(q^{-\frac{1}{2}} x_{(1,1)}) \star \xi_0.
\]

On the other hand the non-simple indecomposable does admit a symplectic structure, making it absolutely $\sigma$-stable. The symplectic wall-crossing formula then takes the form
\[
\mathbb{E}_q(x_1) \star \xi_0 = \mathbb{E}_q(x_{-1})  \cdot \left( \mathbb{E}_{q^2}(q^{\frac{1}{2}} x_{(1,1)}) \star \xi_0 + \mathbb{E}_{q^2}(q^{-\frac{1}{2}} x_{(1,1)}) \star \xi_{(1,1)}\right).
\]
\end{Ex}

According to Theorem \ref{thm:sdWallCross} the product
\[
\overset{\longleftarrow}{\prod_{\mu\in \mathbb{Q}_{> 0} }} A_{\mu}^{\theta } \star A^{\sigma, \theta }
\]
is independent of $\theta$. We will say that a $\sigma$-compatible stability $\theta$ is $\sigma$-generic if $\mu(d) = \mu(d^{\prime})$ implies $\langle d, d^{\prime} \rangle =0$ and if $d$ is a semistable dimension vector of slope zero, then $d \in \Lambda_Q^{\sigma, +}$, \textit{cf.} \cite{joyce2012}, \cite{mozgovoy2013} in the ordinary case. For such $\theta$ the DT series $A_{\mu}^{\theta }$ encodes the slope $\mu$ motivic DT invariants through the factorization
\begin{equation}
\label{eq:dtFactorization}
A_{\mu}^{\theta } = \prod_{\mu(d) = \mu } \prod_{n \in \mathbb{Z}} \mathbb{E}_q((-q^{\frac{1}{2}})^n x_d)^{(-1)^n \Omega_{d,n}^{\theta}}.
\end{equation}
See \cite{dimofte2011}, \cite{kontsevich2011}. We would like to have an analogue of equation \eqref{eq:dtFactorization} in which the orientifold DT invariants are defined by factorizations of $A^{\sigma, \theta }$. Theorem \ref{thm:sdWallCross} would then give a wall-crossing formula for these invariants.

Let $(Q,\sigma)$ be of Dynkin type $A$ or a disjoint union of a quiver of Dynkin type $ADE$ with its opposite; all other finite type quivers with involutions are disjoint unions of these. Then $\Omega^{\theta}_{d,n} =0$ if $n \neq 0$ so we write $\Omega^{\theta}_d$ for $\Omega^{\theta}_{d,0}$. In fact $\Omega_d^{\theta}$ is non-zero only if $d$ is a positive root of the root system attached to $Q$, in which case it is zero or one depending on $\theta$. Note that all stable representations are absolutely stable and that any duality structure is equivalent to either the orthogonal or symplectic duality.

Let $U$ be as above and let $M$ be an absolutely $\sigma$-stable representation with $\Dim \, M = e$ and $r$ regular summands. Proposition \ref{prop:sdStabAut} implies $Aut_S(M) \simeq \mathbb{Z}_2^r$. A modification of the previous calculations shows the contribution of $\{ H(U^{\oplus n}) \oplus M \}_{n \geq 0}$ to $\int_{\mathcal{M}} \mathbf{1}^{\sigma}$ to be
\[
A^{\sigma}_{U,M} = \frac{1}{2^r} \mathbb{E}_{q^2}(q^{\frac{1}{2} - \chi(e,d) - \mathcal{E}(d)} x_{d})\star  \xi_{e}.
\]
Since $\mathcal{E}(d)$ is one or zero depending on whether $U$ does or does not admit a self-dual structure, respectively, this formula specializes to those derived above. Varying $U$ and $M$ over all ($\sigma$-)stable representations (including the $2^r$ $\mathbb{F}_q$-forms of $M \otimes_{\mathbb{F}_q} \overline{\mathbb{F}}_q$) shows that we can write
\begin{equation}
\label{eq:oridtFactorization}
A^{\sigma, \theta } = \sum_{e \in \Lambda_Q^{\sigma, +}} A_{\mu=0}^{\theta } (q, \{ q^{\frac{1}{2} - \chi(e,d) - \mathcal{E}(d) } x_d\}_{d} ) \star \Omega_{e}^{\sigma, \theta} \xi_{e}
\end{equation}
for some non-negative integers $\Omega_{e}^{\sigma, \theta}$. We summarize our calculations as follows.

\begin{Thm}
\label{thm:ftOriDt}
Let $Q$ be a finite type quiver with involution and let $\theta$ be a $\sigma$-generic stability. The orientifold DT series $A^{\sigma, \theta }$ admits a factorization of the form \eqref{eq:oridtFactorization}. Explicitly,
\begin{enumerate}
\item for hyperbolic duality structures (disjoint unions, symplectic representations of $A_{2n+1}$ and orthogonal representations of $A_{2n}$) $\Omega_e^{\sigma, \theta} = \delta_{e, 0}$ for all $e \in \Lambda_Q^{\sigma, +}$, and
\item for non-hyperbolic duality structures $\Omega_e^{\sigma, \theta} =1$ if $e=0$ or $e= e_1 + \cdots + e_k$ for pairwise distinct $e_i \in \Lambda_Q^{\sigma, +}$ with $\Omega_{e_i}^{\theta}=1$. Otherwise $\Omega_e^{\sigma, \theta} =0$.
\end{enumerate}
\end{Thm}

\begin{Rems}
\begin{enumerate}
\item The invariant $\Omega_e^{\sigma, \theta}$ defined by equation \eqref{eq:oridtFactorization} is equal to the stacky number of absolutely $\sigma$-stable self-dual $\mathbb{F}_q$-representations of dimension vector $e$. 
Alternatively, $\Omega_e^{\sigma, \theta}$ is equal to the Euler characteristic of the moduli space of $\sigma$-stable representations of dimension vector $e$, which is either empty or consists of single point.

\item The invariants $\Omega_e^{\sigma, \theta}$ satisfy the primitive wall-crossing formula \eqref{eq:oriWCF}. The most basic instance of this is seen in the $A_2$ example above.

\item The form of equations \eqref{eq:dtFactorization} and \eqref{eq:oridtFactorization} reflect the difference between ordinary and $\sigma$-Jordan-H\"{o}lder filtrations: in the latter there is only one self-dual factor, leading to the linear structure of equation \eqref{eq:oridtFactorization}.

\item The factorization of $A^{\sigma}$ by equations \eqref{eq:dtFactorization} and \eqref{eq:oridtFactorization} encodes all ordinary and orientifold DT invariants and can be regarded as a way of extracting the pure orientifold invariants from $A^{\sigma}$. This is similar to the definition of BPS invariants in topological string theory with orientifolds, \textit{cf.} \cite{sinha2000}, \cite{bouchard2004}, \cite{walcher2009}, where the free energy is decomposed into its ordinary and orientifold contributions.
\end{enumerate}
\end{Rems}

\subsection{Quivers with potential}
\label{sec:quivPot}
We consider briefly the extension of the Hall module formalism to quivers with potential. A potential is an element $W \in kQ \slash [kQ,kQ]$ and a representation of $(Q,W)$ is a finite dimensional module over the Jacobian algebra  $J_{Q,W} = kQ \slash \langle \partial W \rangle$. For each $d \in \Lambda_Q^+$ the potential induces a trace function $w:R_d \rightarrow k$. Given a duality structure on $Rep_k(Q)$, the potential $W$ is called $S$-compatible if its trace $w$ is $S$-invariant. In this case there is an induced duality structure on the abelian category of finite dimensional $J_{Q,W}$-modules. As the homological dimension of this category is generally greater than one, Hall algebra techniques cannot be applied directly to study its DT theory. Instead, we use the equivariant approach of Mozgovoy \cite{mozgovoy2013c}.

Suppose we are given a weight map $\wt : Q_1 \rightarrow \mathbb{Z}_{\geq 0}$. This defines a $k^{\times}$-action on $R_d$ as follows. Given $M \in R_d$ and $t \in k^{\times}$, the representation $t \cdot M$ has the same underlying vector space as $M$ but with structure maps $t^{\wt(\alpha)} m_{\alpha}$. Assume that $W$ is homogeneous of weight one with respect to $\wt$, that is, $w(t\cdot M) = t w(M)$.  If $Q$ has an involution $\sigma$, we additionally assume that $\wt$ is $\sigma$-invariant. This implies that $R_d^{\sigma} \subset R_d$ is $k^{\times}$-stable.

\begin{Ex}
The quiver for $\mathbb{C}^3$ is a single node with three loops $\alpha, \beta, \gamma$ and potential $W = \alpha \beta \gamma - \alpha \gamma \beta$. Give $\alpha$ weight one and the other arrows weight zero. Consider the trivial involution and fix a duality structure. Then $W$ is $S$-compatible if and only if $\tau_{\alpha} \tau_{\beta} \tau_{\gamma} = -1$. Self-dual representations describe $\mathcal{N}=4$ or $\mathcal{N}=2$ supersymmetric gauge theories on the worldvolume of $D3$-branes placed on $O3$- or $O7$-planes. These are gauge theories with orthogonal or symplectic gauge groups and matter in the symmetric or exterior square of the defining representation. More generally, examples arise from quivers with potential arising from consistent brane tilings that admit an orientifold action, such as the conifold and $\mathbb{C}^3 \slash \mathbb{Z}_3$ quivers \cite{franco2007}.
\end{Ex}

Let $k = \mathbb{F}_q$. The equivariant Hall algebra \cite{mozgovoy2013c} is the subalgebra $\mathcal{H}_Q^{eq} \subset \mathcal{H}_Q$ spanned by elements $f= \sum_U a_U [U]$ satisfying $a_U = a_{t\cdot U}$ for all representations $U$ and $t \in \mathbb{F}_q^{\times}$. For each $t \in \mathbb{F}_q$, denote by $f_t = \sum_{w(U) = t} a_U [U]$. In \cite[Proposition 5.12]{mozgovoy2013c} it was shown that the map
\[
\int_{\mathcal{H}}^{eq}: \mathcal{H}_Q^{eq} \rightarrow \hat{\mathbb{T}}_Q, \;\;\;\;\; f \mapsto \int_{\mathcal{H}} f_0 - \int_{\mathcal{H}} f_1
\]
is an algebra homomorphism. Analogously, we define the equivariant Hall module $\mathcal{M}_Q^{eq}$, a $\mathcal{H}_Q^{eq}$-submodule of $\mathcal{M}_Q$, and an equivariant integration map $\int_{\mathcal{M}}^{eq}: \mathcal{M}_Q^{eq} \rightarrow \hat{\mathbb{S}}_Q$, a $\int_{\mathcal{H}}^{eq}$-morphism. Define the orientifold DT series of a quiver with $S$-compatible potential and $\sigma$-compatible stability by
\[
A^{\sigma,\theta } = \int_{\mathcal{M}}^{eq} \mathbf{1}^{\sigma, \theta } \in \hat{\mathbb{S}}_Q.
\]
As in \cite{mozgovoy2013c}, this definition is motivated by the approach to DT theory via motivic vanishing cycles \cite{behrend2013}, extended to non-generic stabilities. 

Repeating the proofs from the sections above with equivariant instead of ordinary integration maps, we find a recursive expression for $A^{\sigma,\theta }$ in terms of $A_d$ and $A_d^{\sigma}$ and a wall-crossing formula relating the DT series $\{A^{\theta }_{\mu} \}_{\mu \in \mathbb{Q}_{> 0} }$ and $A^{\sigma,\theta }$ for different $\sigma$-compatible $\theta$.

\footnotesize

\bibliographystyle{plain}
\bibliography{mybib}
 
\end{document}